\documentclass[11pt,final]{amsart}
\usepackage[pctex32]{graphics}
\usepackage{amsfonts}
\usepackage{amssymb,amscd,latexsym}
\usepackage{amsmath}
\usepackage{epsfig}
\usepackage{color}
\usepackage[normalem]{ulem}
\usepackage{color}

\usepackage[dvipsnames]{xcolor}
\usepackage{fancyvrb}   
 
\definecolor{airforceblue}{rgb}{0.36, 0.54, 0.66}
\definecolor{bleudefrance}{rgb}{0.19, 0.55, 0.91}
\definecolor{darkorchid}{rgb}{0.6, 0.2, 0.8}
\definecolor{darkorange}{rgb}{1.0, 0.55, 0.0}
\definecolor{darkspringgreen}{rgb}{0.09, 0.45, 0.27}

\usepackage[normalem]{ulem}
\usepackage[utf8]{inputenc}
\usepackage[english]{babel} 
\usepackage{enumerate}
\usepackage{hyperref}  
                 \hypersetup{ pdfborder={0 0 0}, 
                              colorlinks=true, 
                              linkcolor=blue, 
                              urlcolor=blue,
                              citecolor=blue,
                              linktoc=page,
                              pdfauthor={Michele Bolognesi, Francesco Russo, Giovanni Staglian{\`o}}, 
                              pdftitle={Some loci of rational cubic fourfolds}
                            }                         
\usepackage{verbatim}

\newtheorem*{thmn0}{Theorem}
\newtheorem*{thm27}{Theorem 2.7}

\newtheorem*{thm34}{Theorem 3.4}
\newtheorem*{cor35}{Corollary 3.5}
\newtheorem*{thm37}{Theorem 3.7}

\newtheorem{thm}{Theorem}[section]
\newtheorem{cor}[thm]{Corollary}
\newtheorem{prop}[thm]{Proposition}
\newtheorem{Proposition}[thm]{Proposition}
\newtheorem{defin}[thm]{Definition}

\newtheorem{lema}[thm]{Lemma}
\newtheorem{rmk}[thm]{Remark}

\newtheorem{Remark}[thm]{Remark}
\theoremstyle{definition}
\newtheorem{ex}[thm]{Example}

\newcommand{\PP}{{\mathbb{P}}}

\newcommand{\map}{\dasharrow}
\newcommand{\codim}{\operatorname{codim}}
\newcommand{\red}{\operatorname{red}}

\newcommand{\Tor}{\operatorname{Tor}}
\def\p{\mathbb P}

\def\P{\mathbb{P}}
\def\I{\mathcal I}
\renewcommand{\O}{\mathcal O}
\def\rk{\operatorname{rk}}

\renewcommand{\H}{\operatorname{H}}
\def\Sing{\operatorname{Sing}}

\newcommand{\mult}{\operatorname{mult}}

\newcommand{\Bl}{\operatorname{Bl}}

\newcommand{\length}{\operatorname{length}}

\renewcommand{\Vert}{\operatorname{Vert}}
\begin{document}

\title{Some loci of rational cubic fourfolds}

\author[M. Bolognesi]{Michele Bolognesi}
\address{Institut Montpellierain Alexander Grothendieck \\ %
Universit\'e de Montpellier \\ %
Case Courrier 051 - Place Eug\`ene Bataillon \\ %
34095 Montpellier Cedex 5 \\ %
France}
\email{michele.bolognesi@umontpellier.fr}

\author[F. Russo]{Francesco Russo*}
\address{Dipartimento di Matematica e Informatica, Universit\` a degli Studi di Catania, Viale A. Doria 5, 95125 Catania, Italy}
\email{frusso@dmi.unict.it}

\author[G. Staglian\` o]{Giovanni Staglian\` o} 
\address{Dipartimento di Ingegneria Industriale e Scienze Matematiche, Universit\`a Politecnica delle Marche, Via Brecce Bianche, 60131 Ancona, Italy}
\email{stagliano@dipmat.univpm.it}
\thanks{*Partially  supported  by the PRIN ``Geometria delle variet\`{a} algebriche" and by the Research Network Program GDRE-GRIFGA; the author is a member of the G.N.S.A.G.A.}

\begin{abstract}
	In this paper we investigate the divisor $\mathcal C_{14}$ inside the moduli space of smooth cubic hypersurfaces in $\PP^5$, whose general element is a smooth cubic containing a smooth quartic rational normal scroll. By showing that all degenerations of quartic scrolls in $\PP^5$ contained in a smooth cubic hypersurface are surfaces with one apparent double point, we prove that every cubic hypersurface contained in $\mathcal C_{14}$ is rational. Combining our proof with the Hodge theoretic definition of $\mathcal C_{14}$, we deduce that on a smooth cubic fourfold every class $T\in \H^{2,2}(X,\mathbb{Z})$  with $T^2=10$ and $T\cdot h^2=4$ is represented by a (possibly reducible) surface of degree four which has one apparent double point.  As an application of our results and of the construction of some explicit examples,	we also prove that the Pfaffian locus is not open in $\mathcal C_{14}$. 
	
\end{abstract}

\maketitle

\section*{Introduction}

Cubic hypersurfaces in $\PP^5$ are among the most studied objects in algebraic geometry. This is surely due to the wealth of geometry that they carry along and, possibly, to the fact that they are somehow a very simply defined class of geometric objects, whose rationality is not yet well-understood. The study of the moduli space $\mathcal C$ of smooth cubic fourfolds, particularly through GIT and the period map, has known some very striking advances in recent years, see for example  \cite{voisin,laza,Loij}, and this analysis  has been developed  in parallel to the study of rationality. 
In particular, Hassett described a countable infinity of divisors $\mathcal C_d$ that parametrize \it special cubic 4-folds, \rm that is \rm cubic hypersurfaces  containing  a surface not homologous to a complete intersection, see \cite{Has00}. One expects that rational cubic fourfolds should be strictly contained in the union of these special divisors $\mathcal C_d$ (more precise conjectures have been formulated by Kuznetsov and Hassett, see \cite[Section 3]{Hasurvey} for the state of the art on the subject).

An interesting  and well studied locus of rational cubics is given by Pfaffian cubics, \it i.e. \rm cubic hypersurfaces admitting an equation  given by the Pfaffian of a $6 \times 6$ anti-symmetric matrix of linear forms. These cubics form a dense set in the Fano-Hassett divisor $\mathcal C_{14}$, which is not open in $\mathcal C_{14}$, as we shall show in Theorem \ref{notpfaff} and Remark \ref{morph}.  
As a  {\it special surface} for  $\mathcal C_{14}$ one can take  either  a smooth quartic rational normal scroll or a smooth quintic del Pezzo surface but also, for instance, the isomorphic projection of a smooth surface of degree 8 and sectional genus 3 in $\p^6$, see \cite[Theorem 2]{RS}. The Pfaffian cubics form a subset in $\mathcal C_{14}$, which consists exactly of  cubic fourfolds  containing a smooth quintic del Pezzo surface, see \cite[Proposition 9.2]{bove-det}.

Other examples of rational cubic fourfolds are given by a countable infinity of irreducible divisors $W_i$ in $\mathcal C_8$, the divisor of cubic 4-folds containing a plane. The families $W_i$'s are thus of codimension two in $\mathcal C$ and consist of cubic 4-folds containing a plane $P$ such that the natural quadric fibration obtained by projection from $P$ has a rational section, yielding directly the rationality of these cubic hypersurfaces.  As first remarked in \cite{ABBVA} and as we shall also show in the last section, there exist rational cubic hypersurfaces in $\mathcal C_8$ such that the associated quadric fibration has no section, see Remark \ref{ratsec} for details. Another countable union of divisors in $\mathcal C_{18}$ parametrizing rational cubic fourfolds has been recently constructed in \cite{AHTV-A}, by showing that these cubics are birational to fibrations of del Pezzo sextic surfaces that admit a section. More recently, the second and third named author proved that a general cubic fourfold in $\mathcal C_{26}$ and in $\mathcal C_{38}$ is rational. This confirms, for these two divisors, the expectations of Kuznetsov conjecture on the rationality of cubic fourfolds (see \cite{RS} for more details). Up to now, the general members of these countable loci in $\mathcal C_8$ and $\mathcal C_{18}$ together with the general elements of $\mathcal C_{14}$, $\mathcal C_{26}$ and $\mathcal C_{38}$ are the only known examples of rational cubic fourfolds.

So far, a direct proof of the rationality of all the elements of $\mathcal C_{14}$ was not known. \footnote{A long 
time after this paper was posted on arXiv, Kontsevich and Tschinkel \cite{KT17} showed that for a smooth proper family over a smooth connected curve the rationality of the generic fiber implies the rationality of all the fibers,
that is rationality specializes.} 
Moreover, $\mathcal C_8\cap\mathcal C_{14}\neq\emptyset$ and  $\mathcal C_{14}$ intersects also many other
divisors $\mathcal C_d$ for which the general member is not known to be rational. 
The geometric description of the divisor $\mathcal C_{14}$ allows us to extend the known rationality of a general element of $\mathcal C_{14}$ to each element of the family  using generalized 
one apparent double point surfaces, dubbed OADP surfaces in the following (see Section \ref{Prel} for precise definitions). In fact, the mere existence of an OADP surface inside a cubic 4-fold  implies its rationality.
One of the  main results of this paper  is the following.

\begin{thmn0}\label{intro-rat}
	All the cubic 4-folds contained in the irreducible divisor $\mathcal C_{14}$ are rational.
	\end{thmn0}

The way to the proof of the previous Theorem requires a complete understanding  of the details of  some results proved by Fano in \cite{Fano}, of which we give new
and modern formulation (and proofs).
Some of these results have been frequently cited and used in the literature but a modern detailed account  does not seem to have appeared elsewhere. More precisely, they rely on the study of the rational map defined by quadrics vanishing on a smooth quartic rational normal scroll, on its restriction to a cubic hypersurface
containing the scroll and on the family of quartic rational normal scrolls contained in a cubic fourfold in $\mathcal C_{14}$. In particular, we are able to prove the next result.
\begin{thm27}
	Let $X$ be a cubic fourfold contained in $\mathcal C_{14}\setminus\mathcal C_8$ and let $\mathcal{T}$ be the Hilbert scheme  of quartic rational normal scrolls contained in $X$. Then $\dim(\mathcal{T})=2$ and each
	point of $\mathcal T$ correspond to a smooth quartic rational normal scroll contained in $X$.
	\end{thm27}

Then we consider the intersection of $\mathcal C_8$ with $\mathcal C_{14}$. The general associated cubic fourfolds contain a plane and (the class of) a smooth quartic scroll. By \it the class \rm of a smooth quartic scroll, we mean a 2-cycle in $A(X)$ with the same intersection theoretical properties as a smooth quartic scroll.
The components of $\mathcal C_8\cap\mathcal C_{14}$ had already been described in \cite{ABBVA} but our approach  here is more direct and different. In fact, without relying on the arithmetic of the intersection lattices involved, we exhibit explicit examples of cubics contained in the components of the intersection $\mathcal C_8\cap\mathcal C_{14}$ and study the degenerations of such cubics.  In particular, we check which components intersect the Pfaffian locus $\mathcal Pf$ of the moduli space, and we prove that \it the Pfaffian locus is not open in $\mathcal C_{14}$. \rm
These results are  summarized in the next theorem.

\begin{thm34}
	 There are five irreducible components of $\mathcal C_8\cap \mathcal C_{14}$. 
The components are indexed by the value $P\cdot T\in\{-1, 0, 1, 2, 3\}$, where $P\subset X$ is a plane and $T$ the class of a small OADP surface such that $T^2=10$ and $T\cdot h^2=4$. 
\end{thm34}

Voisin proved in \cite{voisin} that, for an arbitrary cubic fourfold $X\subset\p^5$, every class $P\in \H^{2,2}(X,\mathbb Z)=\H^4(X,\mathbb Z)\cap \H^2(\Omega^2_X)$ with $P\cdot h^2=1$ and $P^2=3$ is represented by a  plane in $X$. Theorem \ref{ratC14} and Theorem \ref{compC148} yield the following analogous statement.

\begin{cor35} Let $X\in\mathcal C_{14}$ and let $T\in \H^{2,2}(X,\mathbb Z)$ with $T\cdot h^2=4$ and $T^2=10$. Then $T$ is represented by a small OADP surface contained in $X$.
\end{cor35}

Generically, the cycle $T$ above is a smooth quartic rational normal scroll with the exception of the component where $P\cdot T=-1$, where the general element does not contain any smooth quartic rational normal scroll nor any smooth quintic del Pezzo surface and $T$ is the union of two quadric surfaces intersecting along a line. This yields the following result.

\begin{thm37}
	The set $\mathcal{P}f\subset \mathcal C_{14}$ is not open in $\mathcal C_{14}$. Analogously, the set of smooth cubic fourfolds containing a smooth quartic rational normal scroll is not open in $\mathcal C_{14}$.
\end{thm37}

\subsection{Description of the contents}

In Section 1 we develop some technical results that will be used in the rest of the paper. In particular, we explain the relation between certain varieties defined by quadratic equations  and the OADP condition. 

In Section 2 we reconstruct, state in modern terms and prove  Fano's deformation argument and then  we show the rationality of every element of $\mathcal C_{14}\setminus\mathcal C_8$. In Section 3 we describe the components of $\mathcal C_8\cap\mathcal C_{14}$,  prove the rationality of every cubic in this set, discuss their geometry and analyze their intersections with the Pfaffian locus. Finally, we give a quick proof of the non-openness of the Pfaffian locus.

The paper ends with a section containing some examples of cubic hypersurfaces in $\mathcal C_8 \cap \mathcal C_{14}$ crucial for the proof of the non-openness of the Pfaffian locus.
\medskip

{\bf Acknowledgements}. 
We have received support from the Research Network Program GDRE-GRIFGA, by PRIN {\it Geometria delle Variet\` a Algebriche} and by the Labex LEBESGUE. We would like to thank: A. Auel (a comment of whom inspired Thm. \ref{notpfaff}), M. Bernardara, C. Ciliberto, B. Hassett, A. Kuznetsov, D.  Markushevich and  H. Nuer for stimulating conversations and exchange of ideas. We heartfully thank the referees for their careful reading and for many suggestions, which lead to a significant improvement of the exposition.

\section{Preliminary results}\label{Prel}

\subsection{Small varieties and varieties with one apparent double point.}

We begin by recalling a characterization of 2--regular reduced schemes   from \cite{EGHP}. 

\begin{defin}{\rm 
A non-degenerate scheme $X\subset \p^N$  is 2--{\it regular} in the sense of Castelnuovo--Mumford
if its homogeneous
ideal $I(X)$ is generated by quadratic equations and if $I(X)$ has a linear resolution. In particular
the first syzygies of $I(X)$ are generated by the linear ones.}
\end{defin}
Examples of $2$-regular irreducible varieties are non-degenerate irreducible varieties $X\subset \p^N$
of minimal degree $\deg(X)=\codim(X)+1$, which are also characterized by the previous
algebraic property.

\begin{defin}{\rm 
A scheme $X\subset \p^N$ is {\it small} if for every linear space $P\subset\p^N$ such that
$Y=P\cap X$ has finite length $\deg(Y)$ we have $\dim(\langle Y\rangle)=\deg(Y)-1$, that is the 
$\deg(Y)$ points are linearly independent in $\p^N$.}
\end{defin}

\begin{defin}{\rm 
Let $X=X_1\cup X_2\cup..\cup X_r\subset\p^N$ be a reduced scheme with $X_i$ irreducible and with $X_i\not\subseteq X_j$ for every $i$ and $j$. The sequence of schemes $X_1, X_2, \dots, X_r\subset\p^N$ is {a linearly joined sequence} if
$$(X_1\cup X_2\cup...\cup X_i)\cap X_{i+1}=\langle X_1\cup X_2\cup...\cup X_i\rangle \cap \langle X_{i+1}\rangle$$ 
for every $i=1,\ldots, r-1$, where $\langle Y\rangle$ denotes the linear span of a
scheme $Y\subset\p^N$.}
\end{defin}

One should remark that the previous property depends on the order of the irreducible components, see \cite[Example 0.5]{EGHP}.
\begin{thm}\label{smallschemes}\cite[Thm. 0.4]{EGHP} Let  $X\subset\p^N$ be a reduced scheme.
Then the following conditions are equivalent:
\begin{enumerate}
\item[\rm{(i)}] $X$ is small;
\medskip
\item[{\rm(ii)}] $X$ is 2-regular;
\medskip
\item[\rm{(iii)}] $X$ is a linearly joined sequence of  irreducible varieties of minimal degree.
\end{enumerate}
\end{thm}
\medskip

Let us recall that, given homogeneous forms $f_i$ of degree $d_i\geq 1$, $i=0,\ldots M$, a vector of homogenous forms $(g_0,\ldots,g_M)$
is {\it a syzygy} if $\sum_{i=0}^Mf_ig_i=0$. If $d_1=\cdots=d_M=d$ and if $\deg(g_i)=h$ for every $i=0,\ldots, M$, then  we say that $(g_0,\ldots, g_M)$
is a syzygy of degree $h$ and for $h=1$ we shall say that the syzygy is {\it linear}. For $i<j$ the syzygies $(0,\ldots,0,f_j,0,\ldots,0, -f_i,0,\ldots0)$,
corresponding to the trivial identity $f_if_j+f_j(-f_i)=0$ are called {\it Koszul syzygies}. We say that the Koszul syzygies are generated by 
the linear ones if they belong to the submodule generated by the linear syzygies.

Next we state a result of Vermeire, which applies to 2-regular schemes, but also for example to quintic del Pezzo surfaces in $\p^5$.

\begin{prop}\label{verme} {\rm (\cite[Proposition 2.8]{Vermeire})} Let $f_0,\ldots, f_M$ be homogeneous forms in $N+1$ variables of degree $d\geq 2$
such that the Koszul  syzygies are generated by the linear ones. Then the closure of each fiber of the rational map

$$\phi=(f_0:\ldots:f_M):\p^N\map\p^M$$
is a linear space $\p^s$, which for $s>0$  intersects scheme theoretically the base locus scheme of $\phi$ along a hypersurface of degree $d$.
\end{prop}

Let us introduce the following important definition.
\begin{defin}{\rm 
Let $X$ be an equidimensional reduced scheme in $\p^{2n+1}$ of dimension $n$. The scheme $X$ is called {\it a (generalized) variety with one apparent double point}, briefly {\it  OADP variety}, if through a general point of $\p^{2n+1}$
there passes a unique secant line to $X$, that is a unique line cutting $X$ scheme-theoretically in a reduced length two scheme.}
\end{defin}

 The name $OADP$ variety is usually reserved for the irreducible reduced schemes satisfying the previous condition and it comes from the fact that the projection of $X$ from a general point into $\p^{2n}$ acquires a singular point $p$, which is {\it double}. In fact,  the singular point $p$ arises by collapsing two distinct points $q_1, q_2$ collinear with the center of projection and the  tangent cone at $p$ is the union of  the projections of  the tangent spaces at $q_1$  and $q_2$ so that it consists of two $\p^n$'s intersecting at  $p$.
\medskip

Let $\mathbb G(k,n)$ denote the Grassmannian of $k$-dimensional linear subspaces in $\P^n$. The {\it abstract secant variety $S_X$} of a variety $X\subset\p^{2n+1}$ is the restriction of the universal family of $\mathbb G(1,2n+1)$ to the closure of the image of the rational map which associates to a pair of distinct points
of $X\times X$ the line spanned by them. If $X$ is an OADP variety, by definition the tautological morphism $p:S_X\to\p^{2n+1}$ is birational so that, by Zariski's Main Theorem, the locus of points of $\p^{2n+1}$ through which there passes
more than one secant line has codimension at least two in $\p^{2n+1}$. 

The upshot is that any cubic hypersurface in $\P^{2n+1}$ containing an OADP variety is birational to the symmetric product $X^{(2)}$ if $X$ is irreducible, or to the product of two irreducible components of $X$ if it is reducible (see \it e.g. \rm \cite{Russo}).  By definition, {\it the secant variety} to the  variety $X\subset\p^{2n+1}$ is $SX:=p(S_X)\subseteq\p^{2n+1}$. 

We  define {\it the join $S(X,Y)$} of two reduced schemes $X=X_1\cup\ldots \cup X_r\subset\p^M$ and $Y=Y_1\cup\ldots\cup Y_s\subset\p^M$, with each $X_i$ and $Y_j$ irreducible  for every $i=1,\dots,
r$ and for every $j=1,\dots, s$, by first defining the join of two irreducible components as
$$S(X_i,Y_j)= \overline{\bigcup_{x\neq y ,\, x\in X_i,\, y\in Y_j}\langle x,y\rangle}\subseteq\p^M,$$
and finally letting 
$$S(X,Y)=\bigcup_{i,j}S(X_i,Y_j)\subseteq\p^M.$$
Clearly $\dim(S(X_i,Y_j))\leq\min\{\dim(X_i)+\dim(Y_j)+1,M\}$.
Moreover, with these definitions we have that $SX=S(X,X)$.
\medskip

Let us state an interesting consequence of the two previous results and of the definition of  generalized OADP variety.
\medskip

\begin{cor}\label{genOADP} Let $X\subset\p^N$ be a non degenerate reduced algebraic set scheme-theoretically defined by quadratic forms such that 
their Koszul syzygies are generated by linear syzygies. If through a general point of $\p^N$ there passes a positive finite number of secant lines
to $X$, then $X\subset\p^N$ is a generalized OADP variety.

In particular a small algebraic set $X\subset\p^N$ such that through a general point of $\p^N$ there passes a positive finite number of secant lines to $X$ is a generalized  OADP variety.
\end{cor}

\begin{proof} Let $f_0,\ldots, f_M$ be the quadratic forms defining $X$ and let $\phi:\p^N\map\p^M$ be the associated rational map.  By Proposition \ref{verme} the closure of the fiber of $\phi$ passing through a general point $p\in \p^N$ is 
a positive dimensional linear space $L_p$ containing  all the secant lines to $X$ passing through $p$ ($\phi$ contracts these secant lines to the point $\phi(p)$). Then $L_p\cap X$ is a quadric hypersurface in $L_p$ by Proposition \ref{verme}. Moreover,  $\dim(L_p)=1$ because otherwise
through $p$ would pass infinitely many secant lines to the positive dimensional quadric $L_p\cap X$ and a fortiori to $X$,  contrary to our assumption. In conclusion $L_p$ is the
unique secant line to $X$ passing through $p$.
\end{proof} 

\bigskip

We recall the next result for future reference.

\begin{prop} {\rm (\cite[Proposition 9.1.1, third formula]{Fulton})} Let $X\subset\p^5$ be a smooth cubic hypersurface and let $S_1, S_2\subset X$ be two smooth surfaces
 such that the scheme-theoretic intersection $S_1\cap S_2$ contains a smooth curve $C$ of degree $d$
and genus $g$. Then:
\begin{equation}\label{eq:excess}
\mult_C(S_1\cdot S_2)=3d+K_{S_1}\cdot C+K_{S_2}\cdot C+2-2g,
\end{equation}
where $K_{S_i}$ denotes the canonical class of $S_i$ and $\mult_C(S_1\cdot S_2)$ the multiplicity of intersection
of $S_1$ and $S_2$ along $C$.
\end{prop}

\section{Cubic hypersurfaces in \texorpdfstring{$\mathcal C_{14}$}{C14}}

\subsection{Smooth quartic rational normal scrolls in $\p^5$ and the linear system of quadric hypersurfaces through them}

Let $\mathcal C$ be the moduli space of smooth cubic hypersurfaces in $\p^5$, which is a quasi-projective variety of dimension 20. For generalities on this space see \cite{Has00}.

Let us recall from \cite[Sect. 4]{Has00} that $\mathcal C_{14}\subseteq \mathcal C$ is defined as  the locus   of smooth cubic hypersurfaces $X\subset\p^5$
containing a 2-dimensional algebraic cycle $T$ such that $T^2=10$ and $T\cdot h^2=4$, where  $h$ is the cycle
of a hyperplane section of $X$. 
The locus $\mathcal C_{14}$ is easily seen to be equal to the closure of smooth cubic
hypersurfaces in $\p^5$ containing a smooth rational normal scroll of degree 4.
In fact, if $T\subset X$ is a smooth quartic rational normal scroll, then $T^2=10$  by the self-intersection formula and $T\cdot h^2=\deg(T)=4$. 

The rationality of cubic hypersurfaces in $\p^5$ often depends
on the fact that they contain an OADP surface, irreducible or reducible. Indeed in this case (see \cite[Sect. 5]{Russo} for an extended discussion of the details), the cubic hypersurface is birational to the symmetric product of an irreducible OADP surface or to the ordinary product of two distinct
irreducible components of an OADP surface, whose {\it secant join} fills the whole space.
Examples of surfaces with one apparent double point are: smooth quintic del Pezzo surfaces;
smooth quartic rational normal scrolls and more generally  small varieties whose secant variety (or {\it join}) fills the whole space; the union of two disjoint planes.
A generalization of OADP surfaces has been considered recently in \cite{RS}, providing a new geometric insight to rationality of cubic fourfolds.

There are two types of smooth quartic rational normal scroll surfaces:  $S(2,2)$, projectively generated by two conics in skew planes and isomorphic
to $\mathbb F_0=\p^1\times\p^1$, and $S(1,3)$, projectively
generated by a line and a twisted cubic and isomorphic to $\mathbb F_2$. The first type is the most general one and it depends on 29 parameters
while the second type depends on 28 parameters. The 
application of Proposition \ref{verme} to a smooth quartic rational normal scroll yields the 
following result, which  is  quite well known.

\begin{Proposition}\label{maptoQ} Let $T\subset\p^5$ be a smooth quartic rational normal scroll and let
$\psi:\p^5\map\p^5$ be the rational map defined by the linear system $|H^0(\I_T(2))|$. Then:
\medskip
\begin{enumerate}
\item[{\rm a)}] the closure of the image $Q=\overline{\psi(\p^5)}\subset \p^5$ is a smooth quadric hypersurface;
\medskip
\item[{\rm b)}] the closure of a general fiber of $\psi$ is a secant  line to $T$;
\medskip
\item[{\rm c)}] the closure of a fiber of dimension greater than one is a plane cutting $T$ along a conic.
\end{enumerate}
\end{Proposition}

An explicit birational representation of a smooth cubic hypersurface containing a smooth quartic rational normal scroll $T$ has been described by Fano in   \cite{Fano} 
and by Tregub in \cite{Tregub1}.

In the desire of being as self-contained as possible, we will now provide a short and complete proof of Fano's result.
It is important to point out
that here we consider any smooth cubic hypersurface containing a smooth quartic rational normal scroll, as in \cite[Theorem 4.3]{AR}.
\medskip

\begin{thm}\label{ratFano} {\rm (\cite{Fano}, \cite[Theorem 4.3]{AR})} Let the notation be as in Proposition \ref{maptoQ}. Let $X\subset\p^5$ be a smooth cubic hypersurface containing a smooth quartic rational normal scroll $T$, 
 and let $\widetilde \psi: \Bl_TX\to Q$ be the morphism induced by restricting $\psi$ to $X$.
Then $\widetilde\psi$ is a birational morphism onto a smooth quadric hypersurface  $Q\subset \p^5$ such that  the following properties hold.
\medskip
\begin{enumerate}
\item[\rm{a)}] The [closures of the] positive dimensional fibers of the restriction of $\psi$ to $X$ are either secant (or tangent) lines to $T$ contained
in $X$ or (at most a finite number of) planes cutting $T$ in a conic.
\medskip

\item[\rm{b)}] the inverse map $\widetilde \psi^{-1}:Q\map \Bl_TX$ is not defined along an irreducible surface $S'_T$, whose singular points
are the images of the planes cutting $T$ in a conic and contained in $X$. In particular, the cubic hypersurface
$X$ contains a two dimensional family of secant lines to $T$ and $S'_T$ has at most a finite number of singular points.
\medskip
\item[\rm{c)}] If $X$ does not contain any plane cutting $T$ in a conic, then $S_T'\subset\p^5$ is a smooth surface
of degree 10 and sectional genus 7, which is the projection from a tangent plane of  a smooth K3 surface $S_T\subset \p^8$ of degree 14
and sectional genus 8. The surface $S'_T$ is isomorphic to the Hilbert scheme of secant lines to $T$ contained in $X$.
The conic $C_T\subset S'_T$,
image of the exceptional divisor on the blow-up of $S_T$ via tangential projection, is also the image via $\psi$
of the secant lines to $T$ lying in the planes cutting $T$ in a conic.
\end{enumerate}
\end{thm}

\begin{proof}The scroll $T$ is the base locus of the rational map $\psi:\p^5\map\p^5$, and the general secant line to $T$ is not contained in $X$ and cuts $X$ in one point outside $T$.
Thus the restriction of $\psi$ to $X$ is birational and $\widetilde\psi$ is a birational morphism.   Proposition \ref{maptoQ} implies that   
the positive dimensional fibers of the restriction of $\psi$ to $X$ are exactly as in a).

Let $S'_T\subset Q$ be the fundamental locus of $\widetilde\psi^{-1}$ and let 
$E=\widetilde\psi^{-1}(S'_T)\subset \Bl_TX$. Since $Q$ is smooth,  $E$ is a divisor in $\Bl_TX$ and it is irreducible by \cite[Proposition 1.3]{ESB} because $\Bl_TX$ has Picard group isomorphic to
$\mathbb Z\oplus \mathbb Z$. Then   $S'_T=\widetilde\psi(E)\subset Q$ is irreducible.

Since $X$ contains at most a finite number of planes, the general positive dimensional fiber of $\widetilde \psi$ has dimension one by part a) and $(S'_T)_{\red}$ is an irreducible surface. Let $$Z=\{q\in S'_T\;:\; \dim(\widetilde \psi^{-1}(q))\geq 2\}\subsetneq S'_T$$ and let $V=Q\setminus Z$. Then $\widetilde \psi^{-1}(V)\to V$ is a projective birational morphism between smooth varieties such that each
positive dimensional fiber has dimension at most one. By a result of Danilov, see \cite{Danilov},  $\widetilde \psi^{-1}(V)\to V$ is the blow-up of $V$ along the smooth surface $S'_T\setminus Z$ so that  the base locus scheme of $\widetilde\psi^{-1}$ is smooth outside $Z$. By a straightforward adaptation of \cite[Proposition 2.1 b)]{ESB} we deduce that  $\Sing(S'_T)=Z$  is at most a finite set in bijection with the planes cutting $T$ in a conic  and contained in $X$, proving b).

Let $\Sigma_T\subset\p^5$ be the the union of the planes cutting $T$ along a conic. Since these conics vary in a pencil, the three dimensional variety $\Sigma_T$ is a rational normal scroll
of degree three. Then either $\Sigma_T$ is  a  Segre 3-fold $\P^1\times\p^2\subset \p^5$ (if  $T\simeq S(2,2)$),  or $\Sigma_T$ is  a cone over a twisted cubic
with vertex  a line (if $T\simeq S(1,3)$; in this  case the vertex of the cone is the directrix line of $T$).\footnote{Recall that the only conics living inside $S(1,3)$ are the directrix union a fiber.} 

Let $\Pi$ be a plane meeting $T$ in a (possibly reducible) conic $D$. If $\Pi$ is not contained in $X$, then $\Pi\cap X$ consists of the conic $D$ plus a line $L$, which is  thus secant to $T$.
These lines describe
a rational scroll of degree five $Z_T\subset X$, linked to $T$ via $X$ inside $\Sigma_T$. We claim that the image of $\Sigma_T$ (and hence of $Z_T$) via $\psi$ is a conic $C_T\subset S'_T$.  We shall prove the claim
for $\Sigma_T\simeq \p^1\times \p^2$, the remaining case being similar. The restriction of $\psi$ to $\Sigma_T\simeq\p^1\times \p^2$ is given by a linear system in $|\H^0(\mathcal O(2,2))|$, having  $T$ as base locus
scheme and hence as a fixed component. Since $T$ is a divisor of type $(0,2)$ inside $\p^1\times \p^2$, the restriction of $\psi$ is given by the complete  linear system  $|\H^0(\mathcal O(2,0))|$, proving the claim
(see also Corollary \ref{ST} for a different geometrical incarnation  of the conic $C_T$).

Under the hypothesis of c) one immediately deduces from b)  that $S'_T$ is a smooth irreducible surface and that  the restriction of $\widetilde\psi:\Bl_TX\to Q$  to $E=\widetilde\psi^{-1}(S_T')\to S'_T$
is a $\p^1$-bundle over $S'_T$ whose image in $X$, let us say $M$, is the locus of secant lines to $T$ contained in $X$.  From this it follows that $S'_T$ is isomorphic to
the Hilbert scheme of secant lines to $T$ contained in $X$.
For the geometrical description of
$S'_T$ as the tangential projection of $S_T$ we refer to \cite{Fano} or to  \cite[Theorem 4.3]{AR}, where it is also proved that $M$ is a divisor in $|\mathcal O_X(5)|$
having triple points along $T$. 
\end{proof}
\medskip

\subsection{Singular quadrics through a smooth quartic rational normal scroll.} 
In this paragraph we describe the geometry of quadric hypersurfaces through a quartic rational normal scroll. First we give a synthetic description of how families of quadrics of given rank are constructed and then we collect in a proposition the description of these singular quadrics. Finally we use this to study secant lines of rational normal scrolls contained in a cubic fourfold.

\subsubsection{Rank 4 quadrics}\label{rk4}

Let $T\subset\p^5$ be a smooth quartic rational normal scroll.  The projection of $T$ from a proper secant line $L$, not lying on a plane cutting $T$ in a conic, is a smooth quadric surface $Q\subset\p^3$ and the join $S(L,Q)\subset\p^5$ is a rank 4 quadric 
through $T$. By varying $L$ we get a four dimensional family $\Delta_1$  of rank 4 quadrics through $T$. 

Let $\hat Q$ be a rank four quadric through $T$ and let  $\Vert(\hat Q)=L$ be a line. The projection of $\hat Q$ from $L$ is a smooth quadric $\overline Q\subset\p^3$. Then the projection of $T$ from $L$
is also $\overline Q$. Therefore either $L\subset T$ is a line of the ruling or $\length(L\cap T)=2$ (otherwise the degree of the projection of $T$ from $L$ would be 4, if $L\cap T=\emptyset$; or 3, if
$\length(L\cap T)=1$). So $L$ is a secant or a tangent line to $T$, not contained in a plane cutting $T$ along a conic (otherwise $\overline Q$ would be singular).

\subsubsection{Rank 3 quadrics.}

 We have seen in Sect. \ref{rk4} that, starting from a proper secant line $L$, we obtain a rank 4 quadric $S(L,Q)$ through $T$. When $L$ degenerates to a tangent line to $T$,
including the lines contained in $T$, the projection from $L$ remains smooth. The projection from a secant line $L'$ to $T$ contained in a plane cutting $T$ in a conic $C$ is a rank three quadric $Q'\subset\p^3$. In the degeneration of $L$ to $L'$,
the rank four quadric surface $S(L,Q)$ degenerates into a rank three quadric
$S(L',Q')$,
whose vertex is the plane spanned by the conic $C$. 

\begin{lema}\label{conicvertex}
The vertex of every rank 3 quadric $\hat Q$ through $T$ cuts $T$ along a conic, which for $T=S(1,3)$ is reducible. 
\end{lema}
\begin{proof}
The projection of $T$ from the vertex of $\hat Q$, $\Vert(\hat Q)$, is a conic $\hat C$ so that every line of the ruling of  $T$ cuts the vertex of $\hat Q$ because it cannot dominate $\hat C$.
Then the points of intersection of the lines of the ruling with $\Vert(\hat Q)$ describe a curve $D\subset \Vert(\hat Q)$ of degree at most 2, which is a section of the ruling of  $T$. If $D$ is a conic, then $T=S(2,2)$. If $D$ is a line, then $T=S(1,3)$ and there exist  a twisted cubic $D'$ disjoint from $D$, which is also a section of the ruling.
Since $D'$ projects onto a conic, $D'$ cuts $\Vert(\hat Q)$ in a point $q\not\in D$. The line $L_q$ of the ruling of $T$ passing through $q$ cuts $D$ in a point $q'$. Then $L_q\subset\Vert(\hat Q)$
because $q,q'\in\Vert(\hat Q)$ and  $D\cup L_q$ is a conic contained in $\Vert(\hat Q)$. 
\end{proof}

An explicit and straightforward computation gives
the description of all the  singular quadric hypersurfaces containing $T$, summarized in the following result.
\medskip

\begin{Proposition}\label{singQT} Let $T\subset\p^5$ be a smooth quartic rational normal scroll. The locus of singular
quadric hypersurfaces through $T$ is a degree 6 hypersurface $\Delta\subset\p^5=|H^0(\I_T(2))|$, supported on the union of two quadric hypersurfaces $\Delta_1,\Delta_2\subset \p^5$.
The quadric hypersurface $\Delta_1$ is smooth and it  occurs with multiplicity 2 in $\Delta$ while the quadric $\Delta_2$ has rank 3 and its vertex $P$ is the plane defining the cubic
rational normal scroll $\Sigma_T\subset\p^5$ determined by the pencil of planes cutting $T$ along conics.

The locus of quadrics of rank less than or equal to 4 consists of $\Delta_1\cup P$ while the locus of quadrics of rank 3 is a conic $\Omega\subset\Delta_1\cap \Delta_2$. Note that, if the scroll is $S(1,3)$,
then $P \subset \Delta_1$.
\end{Proposition}
\medskip

Putting together Theorem  \ref{ratFano} and Proposition \ref{singQT} we obtain a different geometrical description of the surface $S_T'$ parametrizing secant lines to $T$ contained in
a smooth cubic fourfold $X\subset \p^5$ through  $T$.
\medskip

\begin{cor}\label{ST} Let notation be as in Proposition \ref{singQT} and Theorem \ref{ratFano}. Let $X\subset\p^5$ be a smooth cubic hypersurface containing a smooth quartic rational normal scroll $T$ and not containing a plane cutting $T$ in a conic.

Then  the closure of the locus of quadrics of rank four containing $T$ and whose vertex is a line contained in $X$ is a smooth surface $\hat S_T\subset \Delta_1$,
 isomorphic
to  the  surface $S_T'\subset\p^5$. Moreover, the image of the conic $C_T\subset S'_T$ under this isomorphism is the
conic $\Omega\subset\Delta_1\cap \Delta_2$.
\end{cor}
\medskip

In the sequel we shall use the next quite  striking result, which was claimed without proof  in \cite{Fano}.
\medskip

\begin{lema}\label{rank3quadric} {\rm (\cite[bottom of page 75/top of page 77]{Fano})} Let $T_1, T_2\subset \p^5$ be two smooth quartic rational normal scrolls intersecting in a 0--dimensional scheme of length 10. Then there exists a unique  quadric hypersurface $W\subset \p^5$ containing $T_1\cup T_2$, whose vertex is either 
a  line $L$ secant to $T_1$ and to $T_2$ or  a plane  $\Pi$ intersecting each $T_i$ along a conic $C_i$ with $C_1\neq C_2$

If $T_1$ and $T_2$ are contained in a  cubic hypersurface $X\subset\p^5$, then $L\subset X$, respectively $\Pi\subset X$.
\end{lema}

\begin{proof} We will first show that there exists a unique  quadric hypersurface  in $\p^5$ through $T_1 \cup T_2$. Then some computations  will exclude the higher rank cases.

 Let $Y=T_1\cap T_2$, which  by hypothesis  is a 0--dimensional scheme  of length 10. We can suppose that $T_1\simeq \p^1\times\p^1$ is embedded
in $\p^5$ by $|\H^0(\O_{T_1}(1,2))|$, that is $T_1\simeq S(2,2)$ (the  case  $T_1\simeq S(1,3)$ is similar and left
to the reader). The quadric hypersurfaces defining $T_2$ restrict to divisors in $|\H^0(\O_{T_1}(2,4))|$ 
and we have the short exact sequence: 
$$0\to \Tor_1^{\p^5}(\mathcal O_{T_1},\mathcal O_{T_2})\to \mathcal I_{T_2,\p^5|T_1}\to \mathcal I_{Y,T_1}\to 0,$$
where  $\Tor_{1}^{\p^5}(\mathcal O_{T_1},\mathcal O_{T_2})$ is supported on $Y$. From $\Tor_{i+1}^{\p^5}(\mathcal I_{T_2,\p^5},\mathcal O_{T_1})=
\Tor_{i}^{\p^5}(\mathcal O_{T_2},\mathcal O_{T_1})=0$ for every $i\geq 1$,  from the 2-regularity of $T_2\subset\p^5$ and from the previous exact sequence, it follows that
$h^1(\I_{Y,T_1}(2,4))=0$. Thus
$Y$ imposes independent conditions to $|\H^0(\O_{T_1}(2,4))|$ (see also \cite[Lemma 1.1]{EHP} for a similar argument). Therefore we have:

$$
5=h^0(\O_{T_1}(2,4))-\deg(Y)=h^0(\I_{Y,T_1}(2,4))\geq h^0(\I_{T_2,\p^5}(2))-h^0(\I_{T_1\cup T_2,\p^5}(2)),
$$

yielding 

\begin{equation}\label{stimaT2}
h^0(\I_{T_1\cup T_2,\p^5}(2))\geq 1.
\end{equation}

Let $\psi=\psi_2:\p^5\map Q\subset\p^5$ be the rational map associated to $T_2$, defined in Proposition \ref{maptoQ}, and let $\varphi$
be the restriction of $\psi$ to $T_1$. The map $\varphi$  is given by a linear system  $|D|\subseteq |\H^0(\mathcal O_{T_1}(2,4))|$ having the length 10  base locus scheme $Y=T_1\cap T_2$. From  $D^2=16-10=6$,  we infer that
 $S=\overline{\varphi(T_1)}\subset Q\subset\p^5$ is an irreducible surface such that $6=\deg(\varphi)\cdot \deg(S)$. Proposition \ref{maptoQ} implies $\deg(\varphi)\leq 2$, yielding $\deg(S)\in\{3,6\}$.
 The surface $S$ is degenerate in $\p^5$ by \eqref{stimaT2} ($\psi_2$ induces a one-to-one correspondence between the quadrics
vanishing on $T_1\cup T_2$ and  the hyperplanes containing $S$). Let $M=\langle S\rangle\subsetneq\p^5$ with $3\leq \dim(M)\leq 4$. Since $S\subset Q\subset\p^5$ and since $\deg(S)\in\{3,6\}$, we deduce $\dim(M)=4$ and  $h^0(\I_{T_1\cup T_2,\p^5}(2))=1$.

 Let $W\subset\p^5$ be the unique quadric hypersurface containing $T_1\cup T_2$. Let  $Q'\subset \p^5$ be a smooth quadric hypersurface and let $Z\subset Q'$ be a smooth surface. Then, letting $h$ be the class of a hyperplane section on $Z$, the self-intersection formula for $Z$ on $Q'$ yields
$$Z^2=7h^2+4h\cdot K_Z+2K_Z^2-12\chi(\O_Z).$$

Suppose $W$ were smooth. The previous formula implies that $T^2_i=8$ as cycles inside $W$. 
From $H^4(W,\mathbb Z)=\mathbb Z\alpha\oplus\mathbb Z\beta$ with $\alpha^2=1=\beta^2$ and $\alpha\cdot \beta=0$, from $\deg(T_i)=4$
and from $T_i^2=8$, we get $T_i=2\alpha+2\beta$. This would imply $T_1\cdot T_2=8$, contrary to our assumption. Thus $W\subset \p^5$ is not of maximal rank.  A computation as in \cite[Proposition 2.2]{EHP} shows that $W$ cannot be of rank 5. Therefore $W\subset\p^5$ is of rank 3 or 4.

Let us suppose first rank($W$)=4, let $L=\Vert(W)$  and let  $W:=S(L,\overline Q)$ with $\overline Q$ a smooth quadric surface. Then $L$ is a secant line to $T_1$ and to $T_2$, see Section \ref{rk4}, and the scheme $T_1\cup T_2$ intersects $L$ in a scheme of length at least four. If $T_1\cup T_2$ is contained in a cubic hypersurface $X$, then 
 the multiplicity of intersection of $L$ with $X$ is at least four and $L$ is contained in $X$.
 
 Suppose  rank($W$)=3 and let $W=S(\Pi,C)$, i.e. $W$ is the  quadric obtained as a cone whose vertex is the plane $\Pi\subset\p^5$ and whose base is a smooth conic  $C\subset \langle C\rangle=\p^2\subset\p^5$ 
 such that  $\Pi\cap\langle C\rangle=\emptyset.$  Then  $\Pi\cap T_i$ is a conic $C_i\subset T_i$ by Lemma \ref{conicvertex} and we  have  $C_1\neq C_2$ because $T_1\cap T_2$ is zero dimensional. 
If  $T_1\cup T_2$  is contained in a cubic hypersurface $X$, then   $C_1\cup C_2\subset \Pi\cap X$ yields $\Pi\subset X$.
\end{proof}
\medskip

\subsection{Fano's construction revisited and rationality of cubics in \texorpdfstring{$\mathcal C_{14}\setminus \mathcal C_8$}{C14--C8}}

Let  $\mathcal{T}$ be the Hilbert scheme of  (degenerations of) smooth quartic rational normal scrolls contained in a general $X \in \mathcal{C}_{14}$. In order to calculate the dimension of $\mathcal C_{14}$ we need to estimate the dimension of $\mathcal{T}$. Let us recall that for any smooth quartic rational normal scroll $T\subset\p^5$ we have $\dim(|H^0(\I_T(3))|)=27$. The Hilbert scheme $\mathcal H$ parametrizing smooth quartic rational normal scrolls in $\p^5$, is irreducible, generically smooth and it  has dimension 29. Hence   
\begin{equation}\label{parcount}
\dim(\mathcal T)=\dim(\mathcal H)+\dim(|H^0(\I_T(3))|)-\dim(\mathcal C_{14})\geq 56-\dim(|\O_{\p^5}(3)|)=1
\end{equation}
and $\codim(\mathcal C_{14},\mathcal C)=\dim(\mathcal T)-1$.\footnote{Actually the fact that $\codim(\mathcal{C}_{14})\geq 1$ is well known also for Noether-Lefschetz reasons.} We shall immediately prove, that $\dim(\mathcal T)=2$ for a general $X\in\mathcal C_{14}$ without appealing to abstract deformation theory of $T$ inside $X$.

We now come to one of the   gems in Fano's paper \cite[pages 75--76] {Fano}.
As far as we know this geometrical construction has not been yet translated into modern geometrical language despite the great interest that this example has generated over the decades. Let us remark
that, obviously, Fano did not state the next result in this  form.
\medskip

\begin{thm}\label{dimT}   Let $X\in\mathcal C_{14}$, let $T\subset X$ be  a smooth quartic rational normal scroll, 
let $\Sigma_T\subset\p^5$ be the rational normal scroll  given  by the pencil of planes spanned by the conics contained in $T$ and let $\mathcal T$
be the closure of the family of smooth quartic rational normal scrolls contained in $X$ ($i.e.$ the Hilbert scheme of smooth quartic rational normal scrolls contained
in $X$). 

Then:

\begin{enumerate}

\item[{\rm a)}] $\dim(\mathcal T)=2$;
\medskip

\item[{\rm b)}]   there exists a unique   irreducible 2-dimensional component $\tilde S_T\subseteq \mathcal T$ containing $T$, which is  birational to 
the Hilbert scheme of secant lines to $T$ contained in $X$;
\medskip

\end{enumerate}
\end{thm}
\begin{proof} 
By \eqref{parcount} we know that $\dim(\mathcal T)\geq 1$ for any $X\in\mathcal C_{14}$.
Let $L\subset X$ be a proper secant line to $T$, not belonging to the scroll $Z_T$, residual to $T$ in $\Sigma_T\cap X$ (see the proof of Theorem \ref{ratFano} for the definitions). By projecting $T$ from $L$, we deduce that $T\subset W=S(L,\hat Q)\subset \p^5$ with $\hat Q=\pi_L(T)\subset \p^5$ a smooth quadric surface. Let $\Lambda_i=\p^1$, $i=1,2$, be the parameter spaces of the two  ruling of lines contained in $\hat Q$. 

Then, letting

$$\p^3_{\lambda}:=S(L,L_\lambda),\;\;\lambda\in\Lambda_1, \;\;L_\lambda\subset \hat Q;$$
 
 $$\p^3_\mu:=S(L,L_\mu),\;\; \mu\in\Lambda_2, \;\; L_\mu\subset \hat Q,$$
  
  we can define two pencils of cubic surfaces

$$F_\lambda=\p^3_\lambda\cap X,\;\;\lambda\in\Lambda_1$$ and 
$$G_\mu=\p^3_\mu\cap X,\;\; \mu\in\Lambda_2.$$

Modulo a renumbering, we can also suppose that, for general $\lambda\in\Lambda_1$ and
for general $\mu\in\Lambda_2$, we have that
$$\p^3_\lambda\cap T=L'_\lambda\subset T$$ is a line of the ruling of 
$T$ and that $$\p^3_\mu\cap T=C_\mu\subset T$$ is a twisted cubic curve having $L$ has a secant line.

Let $\widetilde L_\mu\subset G_\mu$ be the unique line contained in the smooth cubic surface $G_\mu$
which is skew with $L$ and with $C_\mu$. Let us set
 
$$\widetilde T_L:=\bigcup_{\mu\in\Lambda_2}\widetilde L_\mu\subset X\cap W.$$

Then $\widetilde T_L\subset X$ is a rational  scroll such that $\pi_L(\widetilde L_{\mu})=\pi_L(C_\mu)=L_\mu$ is a line for $\mu$ general. 
By varying the line $L\subset X$ secant to $T$, we can construct a two dimensional family of such surfaces,
 whose general member is a  rational scroll. Among the  secant lines
 of $T$ contained in $X$, there exists a one dimensional family describing the  quintic rational scroll $Z_T\subset X$,
 consisting of secant lines to $T$ contained in a plane meeting $T$ in a conic. Thus a general line $L'$ of the ruling of $Z_T$ is such
 that the corresponding plane $\Pi$ of $\Sigma_T$ is not contained in $X$.
 By degenerating a general secant line $L$ to the secant line $L'$ to $T$, 
 the smooth quadric  $\hat Q=\pi_L(T)$ degenerates  to a rank three
quadric surface $\hat Q'$, whose vertex is  the plane $\Pi$ not contained in $X$. Equivalently, we are degenerating a general quadric
corresponding to a general point in the surface $\hat S_T$, defined in Corollary \ref{ST}, to a quadric corresponding to a point in $\Omega\subset \hat S_T$ such that the vertex of the corresponding rank three quadric is not contained in $X$. Let $C'\subset \Pi$ be the unique conic such that $C'\cup L'=\Pi\cap X$. Then $\Pi\cap T=C'$ and the limits of the  $\p_\mu^3$'s contains $\Pi$.  The two rulings of the smooth quadric $\hat Q=\pi_L(T)$ degenerate into the unique ruling of $\hat Q'$ and  the scroll $\widetilde T_L$  converges to $T$. 

Let us denote by $\tilde S_T\subseteq \mathcal T$ 
the irreducible two dimensional family just constructed, consisting of two dimensional cycles algebraically equivalent to $T$. In particular,   $\dim(\mathcal T)\geq 2$. 
Moreover, for a general secant line $L$ to $T$, the rational scroll  $\widetilde T_L\subset X$ is a smooth quartic rational normal scroll such that $C'_\lambda=\p^3_\lambda\cap \widetilde T_L$ is a twisted cubic having $L$ as a secant line and
such that $\widetilde L_\mu=\p^3_\mu\cap \widetilde T_L$ is the line defined above. Thus $\widetilde T_L$ and  $T$ have opposite behavior with respect to the intersection with the two pencils $\{\p^3_\lambda\}_{\lambda\in\Lambda_1}$ and $\{\p^3_\mu\}_{\mu\in\Lambda_2}$.

 Let $T_2\in\mathcal T$ be a general element in an irreducible component $\mathcal T'$ of $\mathcal T$ to which $T$ belongs. From $10=T^2=T\cdot T_2$ and by the generality of $T_2$
 we deduce that $T$ and $T_2$ intersect in a 0--dimensional  scheme of length 10. By applying Lemma \ref{rank3quadric} to $T$ and $T_2$,  we conclude that $T_2$ is obtained from $T$ by the previous geometrical construction yielding $\mathcal T'=\tilde S_T$ and $\dim(\mathcal T)=2$. Moreover, we also showed that $\tilde S_T$ is the  unique irreducible component of $\mathcal T$ containing $T$, proving the first part of b).
 
Let $S_T'$ be the surface defined in Theorem \ref{ratFano}, which  parametrizes via $\widetilde \psi$ the secant 
lines to $T$ contained in $X$.

The cubic hypersurface $X$ contains at most a  finite number of planes, so for a general $T_2\in \tilde S_T$ the unique quadric hypersurface  $W_{T_2}$ containing $T\cup T_2$ provided by Lemma \ref{rank3quadric}
has rank four. We can define a rational map  $\alpha_T: S'_T\map \tilde S_T$, by associating to a general secant line to $T$ contained in $X$ the scroll $T_2\in \tilde S_T$  produced via Fano's deformation argument of Lemma \ref{rank3quadric}.
Again by Lemma \ref{rank3quadric} this map is birational and the inverse  associates to a general $T_2\in \tilde S_T$ the unique vertex of the rank four quadric $W_{T_2}$ containing $T\cup T_2$. 
\end{proof}
\medskip

\begin{rmk}{\rm 

If  $X$ does not contain any plane of $\Sigma_T$, the surface  $S_T'$ is smooth by Theorem \ref{ratFano}, part b). The  description of Fano's deformation shows that  the secant lines to $T$ contained in $\Sigma_T$ produce the same scroll $T$. Thus the conic $C_T\subset S'_T$ is contracted to a point by $\alpha_T$. Under the previous hypothesis one can show that $\alpha_T$ extends to a morphism, which is the the blow-down of the conic $C\subset S'_T$, and also that $\tilde S_T$ is isomorphic to a smooth $K3$ surface of degree 14 and genus 8, a fact which is known since \cite{BD}.

Addington and Lehn show the existence of a two dimensional family of surfaces of degree four  parametrized by a smooth $K3$ surface $S$ as above for a generic Pfaffian cubic in \cite[Section 2]{AL} via linear algebra, expanding the details of the construction drafted by Beauville
and Donagi in \cite[Remarques (1)]{BD}. For such a generic Pfaffian cubic they prove that each member of the family is an irreducible small surface by exhibiting an  explicit resolution
for each member of the family, see \cite[Section 2]{AL}, and  that a general member of the family is a smooth quartic rational normal scroll. Below we shall generalize this fact to every $X\in\mathcal  C_{14}\setminus\mathcal C_8$
by showing that every surface corresponding to each point of the Hilbert scheme $\tilde S_T$ is a smooth quartic rational normal scroll inside $X$.}
\end{rmk}
\medskip

The next result, which is probably well known to the experts in the field, seems to have not been explicitly stated and/or proved  till now.
As we shall see later in Theorem \ref{notpfaff}, the locus of $X\in\mathcal C_{14}$ containing a quartic rational normal scroll is constructible but not open in $\mathcal C_{14}$.
\medskip

\begin{thm}\label{ratC14} Every $X\in\mathcal C_{14}\setminus (\mathcal C_8\cap\mathcal C_{14})$ contains a  smooth quartic rational normal scroll and hence it is rational. Moreover, the family of  smooth quartic rational normal scrolls contained in such a $X$  is  an equidimensional   projective surface.
\end{thm}

\begin{proof} First we describe the Hilbert scheme of quartic rational normal scrolls via an incidence correspondence over the moduli space of cubics (\bf Step 1 \rm). Then, a degeneration argument shows that 
under our hypothesis every $[T_0]\in  \mathcal T$  is a smooth quartic rational normal scroll (\bf Step 2 \rm).

\medskip

\bf Step 1. \rm Let $\mathcal H$ be the irreducible component of the Hilbert scheme of $\p^5$ whose general member is a  smooth quartic rational normal scroll. Every element $[T]\in\mathcal H$ 
has degree four, dimension 2 and Hilbert polynomial equal to $2t^2+3t+1$, being a  flat projective deformation of a smooth quartic rational normal scroll in $\p^5$.

Let $\widetilde{\mathcal C}\subset\p^{55}=\p(H^0(\mathcal O_{\p^5}(3)))$ be the open subset parametrizing smooth cubic hypersurfaces in $\p^5$. Let
$$\overline{\mathcal C_{14}}=\{([T],[X])\in \mathcal H\times\widetilde{\mathcal C}\;\;:\;\; T\subset X\},$$ let $\pi_1:\overline{\mathcal C_{14}}\to \mathcal H$
be the restriction to $\overline{\mathcal C_{14}}$ of the first projection  and let $\pi_2:\overline{\mathcal C_{14}}\to\widetilde{\mathcal C}$ be the restriction to $ \overline{\mathcal C_{14}}$ of  the second projection.

Recall that $\dim(\mathcal H)=29$ and that for a general $[T]\in\mathcal H$  we have that 
$\pi_1^{-1}([T])$ is an open subset of $\p(H^0(\mathcal I_T(3)))=\p^{27}$. Thus $\overline{\mathcal C_{14}}$ contains an irreducible component $W$ of dimension 56 dominating
$\mathcal H$. Moreover  $W$ is a closed subset of $\mathcal H\times\widetilde{\mathcal C}$ so that  $\pi_2(W)=\mathcal C_{14}$ is a closed irreducible subset of $\widetilde{\mathcal C}$
of dimension 54. In fact, we have seen in Theorem \ref{dimT} that the family of quartic rational normal scrolls contained in a general $[X]\in\mathcal C_{14}$ is two dimensional.

Thus the image of $\overline{\mathcal C_{14}}$ in $ \mathcal C$ is exactly $\mathcal C_{14}$, which is irreducible--a well known fact--and a divisor in $\mathcal C$, see also \cite{Has00}. In particular, for every $[X]\in\mathcal C_{14}$
let $\mathcal T=\pi_1(\pi_2^{-1}([X]))\subset\mathcal H$.\\

\bf Step 2.\rm   We claim that \it if $[X]\in\mathcal C_{14}\setminus (\mathcal C_8\cap \mathcal C_{14})$, then
every $[T_0]\in  \mathcal T$  is a smooth quartic rational normal scroll.\rm

 By definition of $\mathcal H$ we can suppose that there exists a one dimensional family $\{T_{\lambda}\}_{\lambda\in C}$, $[T_\lambda]\in\mathcal H$, where  $C$
is a smooth analytic curve (small disk),  such that  $0\in C$ and $T_\lambda$ is a smooth quartic rational normal scroll  for every $\lambda\in C\setminus 0$. Thus each irreducible component of $T_0$ has dimension two and it is covered by lines, the limits of the lines of the ruling of $T_{\lambda}$. Since $X$ does not contains planes (and hence quadric surfaces) by hypothesis, we deduce that every $[T_0]\in\mathcal T$ is irreducible  and generically reduced.

Let us use the same notation as above. Fix a general point $p\in \p^5$
and let $l_{\lambda}$, $\lambda\neq 0$, be the unique secant line to $T_\lambda\subset\p^5$ passing through $p$. A limit line $l_0$ intersects $T_0\subset\p^5$ along a scheme whose length is at
least two. In particular $T_0\subset\p^5$ is a non-degenerate scheme by the generality of $p$.

If  $T_0$ is reduced,  then $T_0\subset\p^5$ is an irreducible non-degenerate surface of degree four having a secant/tangent line passing through a general point $p\in \p^5$. Thus,  either $T_0$ is a smooth quartic rational normal scroll or it is a cone over a quartic rational normal curve (in the last case the limit line $l_0$ would
necessarily be the line through $p$ and the vertex of the cone). The surface $T_0$ cannot be a cone because in this case
 the tangent space to $T_0$ at its vertex would be $\p^5$, forcing the singularity of $X$ at the vertex of $T_0$. Thus, if $T_0$ is reduced, then it is a smooth quartic rational normal scroll, as claimed.

Suppose $T_0$ is not reduced. The scheme $T_0$ is irreducible and generically reduced and  has  Hilbert polynomial equal to $2t^2+3t+1$ so that $R_0=(T_0)_{\red}$  is  an irreducible surface of degree four in $\p^5$, which is degenerated. Otherwise $R_0$ would be a surface of minimal degree in $\p^5$  which has Hilbert polynomial $2t^2+3t+1$ and it would coincide with $T_0$,
which is non reduced by hypothesis. Moreover, $R_0$ is  covered by lines, which are limits of the lines covering $T_{\lambda}$.  From this it follows that  $R_0$ is an external projection of a quartic rational normal scroll $S_0\subset\p^5$.

If $S_0$ were a cone
over a smooth quartic rational normal curve, then $T_0$ would  be non reduced only at the point $p_0\in R_0$ which is the image of the vertex of $S_0$. The limit secant line $l_0$ through a general point $p\in\p^5$ introduced above would be  necessarily tangent to $T_0$ at $p_0$. The generality of $p$ would yield that the tangent space to $T_0$ at $p_0$ is the whole space $\p^5$, showing that such a $T_0$ cannot be contained in $X$. If $S_0\subset\p^5$ were a smooth
quartic rational normal scroll, then, by Proposition \ref{verme},  $R_0$ is either a general projection of $S_0$ or the external projection of $S_0$ from a point on the scroll $\Sigma$ generated by the planes spanned by the pencil of conics on $S_0$. In the first case, there is a unique non reduced point $p_0\in T_0$ supported
at the unique singular point of $R_0$.  Proceeding as above, we deduce  that $T_{p_0}T_0=\p^5$ and hence $T_0$ cannot be contained in $X$. If $R_0$ is the projection  of $S_0$ from a point of $\Sigma$, then  $T_0$ has embedded points along the  line $L=\Sing(R_0)$. The tangent space at each point of $L$ has dimension four and intersects $\p^4=\langle R_0\rangle$
along a $\p^3$. If such a $T_0$ were contained in $X$, then the intersection $X\cap\langle R_0\rangle$ would be a cubic hypersurface in $\langle R_0\rangle$, containing $R_0$ and  non  singular along $L$. This is impossible, as shown for example by a direct calculation, so  that $T_0$ is necessarily reduced.

In conclusion, the family of smooth quartic rational normal scrolls contained in $X\in\mathcal C_{14}\setminus\mathcal C_8$ is  not-empty and proper so that it coincides with $\mathcal T$.
\end{proof}

\begin{Remark}\label{rmkS} {\rm Under the hypothesis of Theorem \ref{ratC14}, every irreducible component of $\mathcal T$ is a smooth $K3$ surface $S\subset\p^8$ which is a general linear section of $\mathbb G(1,5)\subset\p^{14}.$ 
From our perspective the surface $S$ is obtained in this way: we fix a $T\subset X$ and construct the surface $S'_T$, which is smooth by Theorem \ref{ratFano} and which  parametrizes secant and tangent lines to $T$ contained in $X$; by contracting the conic
$C_T\subset S'_T$ to a point one obtains the surface $S$. 

Let $S^{[2]}$ denote the Hilbert scheme of length two subschemes of the smooth irreducible projective surface $S$. Then $S^{[2]}$ is a smooth irreducible projective variety of dimension 4. One can also describe $S^{[2]}$ as the blow-up of the symmetric
product $S^{(2)}$ along the image of the diagonal $\Delta_S\subset S\times S$, yielding a birational morphism $\varphi:S^{[2]}\to S^{(2)}$. Let $E:=\varphi^{-1}(\Delta_S)$ and let
$(p,p)\in \Delta_S$. Then $\varphi^{-1}((p,p))\simeq\p(t_pS)$, where $t_pS$ is the affine tangent space to $S$ at $p$, i.e. $E$ is the union of the exceptional divisors of the blow-up's
of $S$ at each point $p\in S$.

 By definition of Hilbert scheme,  each point  $p\in S$ corresponds to a unique smooth quartic rational normal scroll $T_p\subset X$ and, by Theorem \ref{ratFano}, the Hilbert scheme of secant lines to $T_p$ contained in $X$ is isomorphic to $\Bl_pS$. Under this isomorphism, the secant lines to $T_p$ contained in $\Sigma_{T_p}$ correspond to the exceptional divisor $\p(t_pS)\simeq\varphi^{-1}((p,p))$.

The Hilbert scheme $F(X)$ of lines contained in $X$
can be interpreted as   the parameter space of   secant lines to the family of smooth rational normal scrolls contained in $X$
in a natural way, yielding a different interpretation of the well known  isomorphism with $S^{[2]}$.

In fact, we can consider $S^{[2]}$ as the Hilbert scheme parametrizing {\it couples} of smooth quartic rational normal scrolls
contained in $X$. If a (general) point $[(p_1,p_2)]\in S^{[2]}$ corresponds to two distinct $T_{p_1},T_{p_2}$ contained in $X$,  we can associate to $[(p_1,p_2)]$ the vertex of the unique rank four quadric surface containing two distinct quartic rational normal scrolls $T_{p_1},T_{p_2}$ (see Lemma \ref{rank3quadric}). This extends to a morphism from $S^{[2]}$ to the Hilbert scheme of lines contained in $X$ in the following way. A length two subscheme $\mathbf T\in\varphi^{-1}((p,p))$ can be seen as a limit of $[(T_p,T_{p'})]\in S^{[2]}$ with $T_p\neq T_{p'}$. Thus 
there exists a unique rank three quadric surface  determined by the degeneration of $T_{p'}$ to $T_p$, see proof of Theorem \ref{dimT}. Then this plane
cuts $T_p$ along a conic and a line $L$. By mapping $\mathbf T$ to $L$ one gets  a morphism $\phi:S^{[2]}\to F(X)$, which is indeed an isomorphism.

This construction in some sense generalizes the same isomorphism obtained in \cite{BD} for Pfaffian cubics. While the proof of Beauville and Donagi is based on a linear algebra argument and on some explicit geometry of the Grassmannian $G(2,6)$, the approach sketched above relies completely on the geometry of quartic scrolls inside $X$ and of their secant lines.}
\end{Remark}

\section{Irreducible components of \texorpdfstring{$\mathcal C_{14} \cap \mathcal C_8$}{C14 cap C8} and the Pfaffian locus in \texorpdfstring{$\mathcal C_{14}$}{C14}}
Let $$A(X)=\H^{2,2}(X,\mathbb Z)=\H^4(X,\mathbb Z)\cap \H^2(\Omega^2_X)$$ denote the lattice of algebraic 2-cycles on the cubic fourfold $X\subset\p^5$ up to rational equivalence and let $d_X$ be the discriminant of the intersection form on $A(X).$

Let $\beta\in \mathbb{Z}$,  let $S\subset \p^5$ be a smooth quintic del Pezzo surface and let $P\subset\p^5$ be a  plane. Let
$$\mathcal C'_{\beta}=\{[X]\in\mathcal C_{14}\cap \mathcal C_8\;:\; X\supset S\cup P\text{ with  }  S\cdot P=\beta\}\subset\mathcal C_8\cap \mathcal C_{14}.$$ 

Let $h^2$ be the class of a smooth cubic surface $W\subset X$,
intersection of $X$ with a general $\p^3\subset\p^5$ and let $\overline T=3h^2-S\in A(X)$. 
Since $h^4=3$, $h^2\cdot S=5$ and
$S^2=13$, we get $\overline T\cdot h^2=4$,  $\overline T^2=10$ and $\overline T\cdot P=3-\beta$. 
  
Let $\tau\in\mathbb Z$, let $T\subset\p^5$ be a  smooth quartic rational normal scroll and let $P\subset\p^5$ be a plane. Let 
$$\mathcal C''_{\tau}=\{[X]\in\mathcal C_{14}\cap \mathcal C_8\;:\; X\supset T\cup P\text{ with  }  T\cdot P=\tau\}\subset\mathcal C_8\cap \mathcal C_{14}.$$
\medskip

\begin{Proposition} \label{Tregubcomp} The set $\mathcal C'_\beta$, respectively   $\mathcal C''_\tau$, defined here above is not empty if
and only if $\beta$, respectively $\tau$, belongs to $\{0,\,1,\,2,\,3\}$.
\end{Proposition}

\begin{proof}
 If $\beta<0$, then the cycle $P\cdot S$ has to contain a curve. Since $S$
is defined by quadratic equations the scheme-theoretic intersection $P\cap S$ is either a line $L\subset S$ or a conic $C\subset S$. From $K_S\cdot L=-1$ and $K_S\cdot C=-2$, we deduce from formula \eqref{eq:excess} that $\mult_L(P\cdot S)=1$, respectively $\mult_C(P\cdot S)=0$, contrary to our assumption. The argument for $\mathcal C''_\tau$ with negative $\tau$ is identical and will be omitted.

The surfaces $T$ and  $S$ are  scheme-theoretically defined by quadratic equations whose first syzygies are generated by the linear ones.
If $P\cap T$, respectively $P\cap S$ is 0-dimensional, then 
$\tau\leq 3$, respectively $\beta\leq 3$, by the last part of \cite[Theorem 1.1]{EGHP2}. If $P\cap T$, respectively $P\cap S$, contains a curve then $\tau$, respectively $\beta$, belongs to $\{0,1,2,3\}$
by the above argument,
concluding the proof of the \it only if \rm part. One may  also prove these facts as in  \cite[Theorem 4]{ABBVA}, via  a different argument using lattice theoretic methods and Riemann bilinear relations.

Example \ref{pfafftregub} proves the \it if \rm part.
\end{proof}

\begin{rmk}\label{C-1}

{\rm The closure of the locus of smooth cubic hypersurfaces in $\p^5$ containing a pair of skew planes is irreducible, has codimension 2,
see \cite{Tregub1}, and it will be indicated by $\widetilde{\mathcal C}_{-1}$. For a  general  $[X]\in\widetilde{\mathcal C}_{-1}$ 
we have $\rk(A(X))=3$ and $d_X=21$. If $T\subset X$ were a smooth quartic rational normal scroll,  respectively
if $S\subset X$ were  a smooth quintic del Pezzo surface, then an easy direct computation shows that $d_X=21$ would imply $[X]\in\mathcal C''_{-1}$, respectively $[X]\in\mathcal C'_{4}$, which is impossible by Proposition \ref{Tregubcomp}. This remark due to Tregub in \cite{Tregub1} has some striking consequences on the topological properties  of the Pfaffian locus, see Theorem \ref{notpfaff} below.

A  general  $[X]\in\widetilde{\mathcal C}_{-1}$ contains cycles $\overline T$ with $\overline T\cdot h^2=4$ and $\overline T^2=10$. In particular, such a $X$ contains the reducible small surface consisting of the union of two general quadrics  each one residual to one of the two skew planes contained in $X$. A general $[X]\in\widetilde{\mathcal C}_{-1}$ contains also cycles $\overline S$ with $\overline S\cdot h^2=5$ and $\overline S^2=13$. In particular  $\widetilde{\mathcal C}_{-1}$ is an irreducible component of  $\mathcal C_8\cap\mathcal C_{14}$.}
\end{rmk}

\begin{ex}\label{comppfaff}{\rm ($\mathcal C'_\beta\neq\emptyset$ for $\beta\in\{0,1,2,3\}$) By a direct computation one shows that there exists a smooth
cubic hypersurface $X\subset\p^5$ containing a smooth quintic del Pezzo surface $S$ and a plane $P$ such that $S\cap P$ is a scheme
of length $\beta$ consisting exactly of $\beta$ reduced points (see Example \ref{pfafftregub}). Let $\mathcal Pf\subset\mathcal C_{14}$
be the subset of Pfaffian cubics, that is cubic hypersurfaces in $\p^5$ admitting an equation given by the Pfaffian of a $6 \times 6$ anti-symmetric matrix of linear forms.
By \cite[Proposition 9.2, part (i)]{bove-det} the set $\mathcal Pf\subset\mathcal C_{14}$ consists exactly of cubic fourfolds containing a smooth quintic del Pezzo surface. Moreover
by \cite[Proposition 9.2, part (ii) ]{bove-det} the closure of $\mathcal Pf$ in $\mathcal C$ is irreducible of dimension 19 and hence it coincides with $\mathcal C_{14}$.
In particular $\mathcal Pf$ is dense in $\mathcal C_{14}$.
}
\end{ex}
\medskip

We are now in position to give an alternative, geometrical and self-contained proof of the main result of \cite{ABBVA}. Moreover,  we shall also show that
every element in $\mathcal C_8\cap\mathcal C_{14}$ is rational, a fact claimed only for the general element of some components in \cite{ABBVA}. This result, together
with Theorem \ref{ratC14}, will prove that every element in $\mathcal C_{14}$ is rational.

In fact, the proof of the rationality of the generic cubic contained in the component with $d_X=32$ proposed in \cite{ABBVA} relied on the openness of the locus of Pfaffian cubics inside the divisor $\mathcal C_{14}$. Since Theorem \ref{notpfaff} will show that the Pfaffian locus is not open, then we also fill in this gap and also simplify some arguments in \cite{ABBVA}. 
\medskip

\begin{thm}\label{compC148} The codimension two locus $\mathcal C_8\cap \mathcal C_{14}$ has five irreducible components. The cubic hypersurfaces contained in each component contain a small OADP surface of degree four and hence they are rational.  The components are indexed by the value $P\cdot T\in\{-1, 0, 1, 2, 3\}$, where $P\subset X$ is a plane and $T$ the class of
a small $OADP$ surface such that $T^2=10$ and $T\cdot h^2=4$.
\end{thm}

 The proof is based on a degeneration argument that shows the following claim: every point in  the Hilbert scheme of quartic rational normal scrolls contained in a fixed $X$ as in the statement corresponds   either to a smooth rational normal scroll or to a small reducible OADP surface. In order to prove this claim, we go through a case by case analysis. Finally, we shall compute the irreducible components by lattice-theoretic arguments and by showcasing explicit examples (contained in Sect. \ref{sect4}).

\begin{proof}

The notation will be as in the proof of Theorem \ref{ratC14}. Let $X\subset\p^5$ be such that $[X]\in\mathcal C_{14}\cap\mathcal C_8$ and recall that $\mathcal{H}$ is the irreducible component of the Hilbert scheme of $\p^5$ whose general
member is a smooth quartic rational normal scroll.  Let $\mathcal T\subset\mathcal H$ be the Hilbert scheme parametrizing schemes $T\subset \p^5$ with Hilbert polynomial
$p(t)=2t^2+3t+1$ which are contained in $X$. 
Since a general cubic hypersurface in $\mathcal C_{14}$ contains a two dimensional family of smooth quartic rational normal scrolls, we deduce that $\dim(\mathcal T)\geq 2$ by semicontinuity.
Moreover, each   $[T]\in \mathcal T$ is a  projective flat degeneration of a smooth quartic rational normal scroll. In particular,   
 each irreducible component of $T$ of dimension two is covered by lines which are the limits of the lines covering a general element in $\mathcal H$, which is a smooth quartic rational normal scroll. Let us remark that a priori, for some particular choice of $X$, every element in $\mathcal T$ might be reducible, see for example Remark \ref{C-1}.
By repeating the same argument via the limit secant line we used in the proof of Theorem \ref{ratC14}, we can conclude that  any $T_0\in\mathcal T$  is a non-degenerate scheme in $\p^5$ such that
through a general point of $\p^5$ there passes a secant/tangent line to $T_0$.\\

\bf Claim: \rm a general element 
$T_0\in\mathcal T$ is either a smooth rational normal scroll or a small reducible OADP surface\\ 

{\bf Suppose that $T_0$ is a reduced scheme.}
Then $T_0$ is a non-degenerate reduced surface of degree four. If $T_0$ is irreducible, then it is a smooth quartic rational normal scroll, as shown in the proof of Theorem \ref{ratC14}. If  $T_0$ is not irreducible, then it is the union of surfaces of degree lower or equal to three, all covered by lines. Thus $T_0$ can be  the union of planes, quadric surfaces and rational scrolls of degree three. Since the Hilbert polynomial of a hyperplane section is $4t+1$, a general hyperplane section of $T_0$ is the union of irreducible rational curves such that two irreducible components  intersect at a point. Therefore the intersection of two irreducible components of $T_0$ occurs along a line. Since $T_0$ is also non-degenerate, the intersection of two irreducible components equals the intersection of the corresponding linear spans (otherwise $\langle T_0\rangle\subsetneq\p^5$). Thus $T_0\subset\p^5$ is a linearly joined sequence of  surfaces of minimal degree and hence a small surface by Theorem \ref{smallschemes}. Then $T_0$ is also an OADP surface by Corollary \ref{genOADP} since through a general point of $\p^5$ there passes a secant/tangent line to $T_0$.

{\bf Assume $T_0\subset\p^5$ is a non-reduced non-degenerate scheme.} If $T_0$ is irreducible, then the same argument as in the proof of Theorem \ref{ratC14} shows that  $T_0$ cannot be generically reduced.  Thus $(T_0)_{\red}$ would be a (possibly reducible) quadric surface $Q_0$ and $T_0$ would define the cycle $2Q_0$ inside $X$. The scheme $T_0$ can be obtained as  a flat projective deformation of a one dimensional family $\{T_\lambda\}_{\lambda\in C}\subset \mathcal H$ of small OADP surfaces consisting of two quadric surfaces $Q'_\lambda, Q''_\lambda$ intersecting along a line. Moreover,  we can also suppose that $Q_0=Q'_\lambda$ for every $\lambda\in C$. In particular, for every $\lambda\neq 0$ the surface $T_\lambda$ is contained in smooth cubic hypersurfaces belonging to $\mathcal C_{14}$. Since $X\in\mathcal C_{14}$ we have  $\pi_1^{-1}(T_0)\neq\emptyset$ so that  $X\in W=\pi_2(\pi_1^{-1}(C))\subset\mathcal C_{14}$, where the $\pi_i$ are the projections as in Thm \ref{ratC14}. Then we can suppose that
$X=X_0$ is the limit of a flat family $\{X_\mu\}_{\mu\in D}$ with $X_\mu\in W$ for every $\mu\neq 0$. Let us point out that  $X_\mu\in\widetilde{\mathcal C}_{-1}$ for every $\mu\neq 0$ since, by construction, a general element of $W$ contains a cycle of the form
$Q'_\lambda+Q''_\lambda$ for some $\lambda\in C\setminus 0$. Therefore there exist cycles $T_\mu=Q'_\mu+Q''_\mu=Q_0+Q''_\mu$ inside $X_\mu$ such that $\rk(\langle h^2,T_\mu\rangle)=2$ and such that $\rk(\langle h^2,Q_0,Q''_\mu\rangle)=3$ for every $\mu\neq 0$. Moreover,  the discriminant 
of $\langle h^2,T_\mu\rangle$ is equal to 14 for every $\mu\neq 0$. If $P_0\subset X=X_0$ is the unique plane such that $h^2=P_0+Q_0$, then $\rk(\langle h^2,P_0, T_\mu\rangle)=\rk(\langle h^2, Q_0, Q''\mu\rangle)=3$ for every $\mu\neq 0$. Then $T_0=2Q_0\in \langle h^2,P_0\rangle$ is in contrast with the semicontinuity of the rank
of $A(X_\mu)$ over $\mathcal C$ (or with the fact that $C_{14}$ is closed in $\mathcal C$). This proves that $\pi_1^{-1}(T_0)=\emptyset$, that is: every cubic hypersurface containing a reduced $T_0$ as above is singular. This fact can be also verified by a long computation, which shows that a general element in $\pi_1^{-1}(T_0)$ has three singular points. In an analogous way, one can prove that a non-reduced $T_0$ must be generically reduced along each irreducible component of $(T_0)_{\red}$.

{ \bf From now on  we can suppose that  $T_0$ is a non-reduced, generically reduced scheme having at least two irreducible components, which are either planes or quadric surfaces or cubic scrolls.}
 If one of its components is a cubic scroll, then there is only another irreducible component which is necessarily  a plane. From this it follows that $(T_0)_{\red}$ would be a linear projection of a small OADP surface consisting of a cubic rational normal
scroll and a plane. Then the same argument used in Theorem \ref{ratC14} shows that  $T_0$ would have an embedded point $p_0$ at the acquired intersection of the two irreducible components with $T_{p_0}T_0=\p^5$ and $X$ would be singular.

Therefore we can assume that each irreducible component of $T_0$ is either a plane or an irreducible quadric surface.  Then, once again, it is not difficult to see that $T_0$ is necessarily a linear projection of a small OADP surface with embedded points at the acquired
 intersections of the irreducible
components of $(T_0)_{\red}$, yielding the  singularity of each cubic hypersurface containing $T_0$. 
\smallskip

In conclusion,  each cubic hypersurface in $\mathcal C_{14}\cap \mathcal C_8$ contains a small OADP surface  $T\subset X$ such that
 $T\cdot h^2=4$ and $T^2=10$, as claimed.
\smallskip

\bf Description of the irreducible components: \rm let $P\subset X$ be a plane such that $\rk(\langle h^2, T, P\rangle)=3$ and let $\tau=P\cdot T$. 

If $T$ is irreducible, then $T$ is a smooth quartic rational normal scroll
so that  $0\leq\tau\leq 3$ by Example \ref{comppfaff} or directly by Theorem \ref{smallschemes}. If $T=S\cup P'$ with $S\subset X$ a cubic rational normal scroll  and $P'\subset X$
a plane we can take $P=P'$. Since $P'\cdot S=0$ by \eqref{eq:excess}, we deduce $P\cdot T=P^2=3$. If every irreducible component of $T$ has degree less than or equal to
2,  then $X$ contains a pair of skew planes and one easily deduces $\tau=-1$.
In conclusion $-1\leq \tau\leq 3$ and these examples exist by Example \ref{pfafftregub}. 

Denote by $A_\tau$ the lattice of rank 3 generated by $\langle h^2,S,P\rangle$ with $\tau=P\cdot S$ as above. We shall indicate by  $\mathcal D_\tau \subset\mathcal C_8\cap\mathcal C_{14}$ the locus of smooth cubic fourfolds such that there is a primitive embedding  $A_\tau\subset A(X)$ of lattices preserving $h^2$. For
$-1\leq \tau\leq 3$ each  $\mathcal D_\tau$ is a nonempty subvariety by Example \ref{pfafftregub} and it  is of pure codimension 2 in  $\mathcal C$ by a
variant of the proof of [Has00, Thm. 3.1.2]. The argument at the end of  the proof of \cite[Theorem 4]{ABBVA} assures that for a general $X\in\mathcal D_\tau$ we
have $A(X)=A_\tau$ and that each codimension 2 locus $\mathcal D_\tau$ is irreducible, showing that  $\mathcal C_8\cap \mathcal C_{14}$ has
five irreducible components.
\end{proof}

Voisin proved in \cite{voisin} that, for an arbitrary cubic fourfold $X\subset\p^5$, every class $P\in \H^{2,2}(X,\mathbb Z)$ with $P\cdot h^2=1$ and $P^2=3$ is represented by a unique plane in $X$. Theorem \ref{ratC14} and Theorem \ref{compC148} yield the following analogous result.

\begin{cor}\label{cycle} Let $X\in\mathcal C_{14}$ and let $T\in \H^{2,2}(X,\mathbb Z)$ with $T\cdot h^2=4$ and $T^2=10$. Then $T$ is represented by a small OADP surface contained in $X$.
\end{cor}

\begin{rmk}\label{ratsec}{\rm 
To every $X\in\mathcal C_8$ one associates a rational fibration in quadric surfaces induced by  the projection of $X$  from a plane $P\subset X$ onto a skew plane.
If this fibration admits a rational section, then $X$ is rational. 

Let $X$ be a general cubic in one of the five irreducible components of $\mathcal C_8\cap\mathcal C_{14}$ and let $d_X$ be the  discriminant of $X$. In \cite[Proposition 5]{ABBVA} it is proved that the natural quadric fibration associated to $X$ admits a rational section if and only if $\tau$ is odd, that is if and only if $P\cdot T\not\in\{0,2\}$. For $\tau\in\{-1,1,3\}$  a general cubic in the corresponding irreducible component admits  a rational quadric fibration with a section and it is thus rational, see \cite{ABBVA}. Since every element in $\mathcal C_{14}$ is rational,
a general cubic hypersurface with $P\cdot T\in\{0,2\}$ is rational although the associated quadric fibration has no rational section. The case $P\cdot T=0$ has been already observed
in \cite{ABBVA}, where an explicit example is also constructed.}
\end{rmk}

Theorem \ref{compC148}, the discussion in Example \ref{comppfaff}  and  Example~\ref{pfafftregub}\eqref{disjointplanes} below have essentially shown  the next result, which is in contrast with the usual common sense according to which the Pfaffian locus should be open in  $\mathcal C_{14}$.

\begin{thm}\label{notpfaff} The set $\mathcal{P}f\subset\mathcal{C}_{14}$ is not open in $\mathcal{C}_{14}.$ Analogously, the set of smooth cubic fourfolds containing a smooth
quartic rational normal scroll is not open in  $\mathcal{C}_{14}$.
\end{thm}

\begin{proof} Let $\widetilde{\mathcal{C}}_{-1}=D_{-1}$ be the irreducible
component of $\mathcal{C}_8\cap \mathcal{C}_{14}$ whose general element is a smooth cubic 
hypersurface $X\subset\p^5$ containing two skew planes $P_1, P_2$. Suppose that $\mathcal{P}f$ were open in $\mathcal{C}_{14}$ and consider its intersection $\mathcal{P}f\cap \widetilde{\mathcal{C}}_{-1}\subset \mathcal{C}_{14}$. By Example~\ref{pfafftregub}\eqref{disjointplanes} we know that $\mathcal{P}f \cap \widetilde{\mathcal{C}}_{-1} \neq \emptyset$. Hence, if $\mathcal{P}f$ were open, then $\mathcal{P}f \cap \widetilde{\mathcal{C}}_{-1}$ would meet the dense subset of cubics $X \in \widetilde{\mathcal{C}}_{-1}$ such that $\rk(A(X))= 3$, but as recalled in Remark \ref{C-1} this is not possible.

 The conclusion
in the second statement follows once again from the existence of a smooth cubic fourfold containing a quartic rational normal scroll and a pair of skew planes,
see Example \ref{pfafftregub}, and by the fact that a general $X\in \widetilde{\mathcal{C}}_{-1}$ does not contain a smooth quartic rational normal scroll, see Remark \ref{C-1} .
\end{proof}

\begin{rmk}\label{morph} \rm We recall that the Pfaffian locus $\mathcal Pf$ and the set of cubics 4-folds containing a smooth quartic rational normal scrolls are images of quasi-projective
varieties via suitable morphisms (see \cite[Sect. 8-9]{bove-det}, respectively the proof of Thm. \ref{ratC14}). Thus, by Chevalley Theorem,
they are constructible and in particular they contain
an open non-empty subset of $\mathcal C_{14}$. In fact, $\mathcal C_{14}$ is the closure of both these two open sets. If these open subsets intersect an irreducible component of $\mathcal C_8\cap C_{14}$,
then the general element of this component is Pfaffian, respectively contains a smooth quartic rational normal scroll. Since  these open sets are purely theoretical, with no precise handy description, it is hard to verify whether they intersect an irreducible component or not.
Thus the statement  that a general element of $D_\tau$, $\tau\in\{0,1,2,3\}$, is Pfaffian  requires a quite delicate analysis and cannot be deduced
by simply exhibiting a Pfaffian cubic in $D_\tau$, cfr. \cite[Section 4]{ABBVA}. The examples constructed in Example \ref{pfafftregub} show that $\mathcal Pf\cap D_\tau\neq\emptyset$
for every admissible $\tau$ and that  the intersection is non-empty also for the set consisting of cubic four-folds containing a smooth quartic rational normal scroll.
\end{rmk}

\section{Some cubic fourfolds containing smooth del Pezzo quintics}\label{sect4}

We shall give explicit examples
of smooth cubic hypersurfaces in $\mathbb{P}^5$
which contain 
a quintic del Pezzo surface $S\subset \mathbb{P}^5$
and a plane intersecting $S$ 
   in either the empty scheme or a set of $1\leq i\leq 3$ linearly independent reduced points.
Furthermore, we will showcase 
an example of smooth cubic hypersurface containing $S$
and two disjoint planes 
obtained as linear spans of two irreducible conics on $S$.
All our computations have been done using {\sc Macaulay2} \cite{macaulay2}.
\subsection{} 
A del Pezzo surface of degree $5$ in $\PP^5$ 
can be parametrized by the map associated to the linear system of all cubic curves in $\mathbb{P}^2$
passing through four points in general position.
We choose such a map $f:\mathbb{P}^2\dashrightarrow\mathbb{P}^5$, 
where the 
points are taken to be $(1,0,0)$, $(0,1,0)$, $(0,1,0)$, $(1,1,1)$, and 
$f$ is defined by  
\begin{equation*}
 (t_0,t_1,t_2)\mapsto ({t}_{0}^{2} {t}_{2}-{t}_{0} {t}_{1} {t}_{2},{t}_{0} {t}_{1} {t}_{2}-{t}_{1}^{2} {t}_{2},{t}_{0}
      {t}_{2}^{2}-{t}_{1} {t}_{2}^{2},{t}_{0}^{2} {t}_{1}-{t}_{0} {t}_{1} {t}_{2},{t}_{0} {t}_{1}^{2}-{t}_{1}^{2}
      {t}_{2},{t}_{0} {t}_{1} {t}_{2}-{t}_{1} {t}_{2}^{2}) .
\end{equation*}
If $x_0,\ldots,x_5$ denote homogeneous coordinates on $\mathbb{P}^5$, 
  then the image $S=\overline{f(\mathbb{P}^2)}\subset\mathbb{P}^5$ is  
the  del Pezzo surface defined 
by the five quadratic forms: 
\begin{gather*}
{x}_{2} {x}_{4}-{x}_{1} {x}_{5}, \quad {x}_{0} {x}_{4}-{x}_{1}
      {x}_{5}-{x}_{3} {x}_{5}+{x}_{4} {x}_{5}, \quad  {x}_{2} {x}_{3}-{x}_{0}
      {x}_{5}, \\ {x}_{1} {x}_{3}-{x}_{1} {x}_{5}-{x}_{3} {x}_{5}+{x}_{4}
      {x}_{5}, \quad {x}_{0} {x}_{1}-{x}_{1} {x}_{2}-{x}_{0} {x}_{5}+{x}_{1} {x}_{5} .
\end{gather*} 
On  $S$ there are five pencils of conics 
 whose linear spans determine five Segre threefolds  
 $\Sigma_i\simeq \mathbb{P}^1\times\mathbb{P}^2\subset\mathbb{P}^5$, $i=0,\ldots,4$. 
 These pencils come as images of pencils  $\sigma_0,\ldots,\sigma_4$ 
 on $\mathbb{P}^2$ under 
the  parametrization $f$, where  
$\sigma_0$ is the pencil 
 of conics passing through the four base points of $f$, and 
 $\sigma_1,\ldots,\sigma_4$ are the pencils
of lines passing through one of these four points.
From this one can  explicitly determine $\Sigma_0,\ldots,\Sigma_4$, 
and it turns out that their homogeneous 
 ideals are generated by 
the following  quadratic forms:
\begin{equation*}
\begin{array}{l}
\Sigma_0 : {x}_{2} {x}_{4}-{x}_{1} {x}_{5},\ {x}_{2} {x}_{3}-{x}_{0} {x}_{5},\ {x}_{1} {x}_{3}-{x}_{0} {x}_{4} ; \\ 
\Sigma_1 : {x}_{2} {x}_{4}-{x}_{1} {x}_{5},\ {x}_{0} {x}_{4}-{x}_{1} {x}_{5}-{x}_{3} {x}_{5}+{x}_{4} {x}_{5},\ {x}_{0} {x}_{1}-{x}_{1} {x}_{2}-{x}_{2} {x}_{3}+{x}_{1} {x}_{5} ; \\
\Sigma_2 : {x}_{2} {x}_{3}-{x}_{0} {x}_{5},\ {x}_{1} {x}_{3}-{x}_{1} {x}_{5}-{x}_{3} {x}_{5}+{x}_{4} {x}_{5},\ {x}_{0} {x}_{1}-{x}_{1} {x}_{2}+{x}_{2} {x}_{4}-{x}_{0} {x}_{5} ; \\
\Sigma_3 : {x}_{0} {x}_{4}-{x}_{2} {x}_{4}-{x}_{3} {x}_{5}+{x}_{4} {x}_{5},\ {x}_{1} {x}_{3}-{x}_{2} {x}_{4}-{x}_{3} {x}_{5}+{x}_{4} {x}_{5},\ {x}_{0} {x}_{1}-{x}_{1} {x}_{2}-{x}_{0} {x}_{5}+{x}_{1} {x}_{5} ; \\
\Sigma_4 : {x}_{2} {x}_{3}+{x}_{0} {x}_{4}-{x}_{2} {x}_{4}-{x}_{0} {x}_{5}-{x}_{3} {x}_{5}+{x}_{4} {x}_{5},\ {x}_{1} {x}_{3}-{x}_{1} {x}_{5}-{x}_{3} {x}_{5}+{x}_{4} {x}_{5}, \\ \qquad {x}_{0} {x}_{1}-{x}_{1} {x}_{2}-{x}_{0} {x}_{5}+{x}_{1} {x}_{5} .
\end{array}
\end{equation*}
Two generic conics belonging to the same pencil on $S$ are irreducible 
and 
the two planes 
obtained as linear spans are disjoint. 
 Two such conics  in $f(\sigma_0)$ are:
\begin{align*}
 C_1 &= V({x}_{2}+{x}_{5},{x}_{1}+{x}_{4},{x}_{0}+{x}_{3},{x}_{3} {x}_{4}+{x}_{3} {x}_{5}-2 {x}_{4} {x}_{5}), \\
 C_2 &= V({x}_{2}+2 {x}_{5},{x}_{1}+2 {x}_{4},{x}_{0}+2 {x}_{3},2 {x}_{3} {x}_{4}+{x}_{3} {x}_{5}-3 {x}_{4} {x}_{5}) . 
 \end{align*} 
We also fix three linearly independent points $q_1,q_2,q_3\in S$, and 
four planes $\Pi_0,\ldots,\Pi_3\subset \mathbb{P}^5$ such that 
$\Pi_0\cap S = \emptyset$ and 
$\Pi_i\cap S = \{q_1,\ldots,q_i\}$ (scheme-theoretically), for $i=1,2,3$.
Explicitly, 
\begin{gather*}
\Pi_0 = V({x}_{2}+{x}_{4},{x}_{1}+{x}_{3}+{x}_{5},{x}_{0}-{x}_{4}), \quad 
\Pi_1 = V({x}_{3}-{x}_{4}+{x}_{5},{x}_{1}+{x}_{2}-{x}_{5},{x}_{0}-{x}_{5}), \\ 
\Pi_2 = V({x}_{3}-{x}_{4},{x}_{1}-{x}_{5},{x}_{0}-{x}_{2}+{x}_{5}), \quad 
\Pi_3 = V({x}_{3}-{x}_{4}+{x}_{5},{x}_{1}-{x}_{5},{x}_{0}-{x}_{2}+{x}_{5}) ,
\end{gather*} 
where 
$q_1 =f(1,1,0) = (0, 0, 0, 1, 1, 0)$,  $q_2 =f(1,0,1) = (1, 0, 1, 0, 0, 0)$ and $q_3 = f(0,1,1) = (0, 1, 1, 0, 1, 1)$. 

In Example~\ref{pfafftregub},
 we exhibit the  promised smooth cubic hypersurfaces in $\mathbb{P}^5$. 
 All of them are basically obtained by choosing randomly cubic hypersurfaces containing
 the given subschemes,  until we get one that is smooth. 
 This approach works well due to the closedness of the discriminant locus in the space of 
 cubic forms on $\PP^5$.
\begin{ex}\label{pfafftregub}
The following 
five cubic forms $F_0,\ldots,F_4$ on $\mathbb{P}^5$ 
define smooth hypersurfaces  containing the quintic del Pezzo surface $S$;
moreover $V(F_i)$ contains $\Pi_i$ for $i=0,\ldots,3$, 
and $V(F_4)$ contains the two disjoint planes $\langle C_1 \rangle$ and $\langle C_2 \rangle$.
\begin{enumerate}[(a)]
\item 
$F_0 = {x}_{0} {x}_{1}^{2}-{x}_{1}^{2} {x}_{2}+{x}_{0} {x}_{1} {x}_{3}+{x}_{1}^{2} {x}_{3}-{x}_{2}^{2} {x}_{3}+{x}_{1} {x}_{3}^{2}-{x}_{0}^{2} {x}_{4}-{x}_{2}^{2} {x}_{4}+{x}_{1} {x}_{3} {x}_{4}-{x}_{2} {x}_{3} {x}_{4}-2 {x}_{2} {x}_{4}^{2}+{x}_{0} {x}_{2} {x}_{5}-{x}_{3}^{2} {x}_{5}+2 {x}_{1} {x}_{4} {x}_{5}+{x}_{4}^{2} {x}_{5}-{x}_{0} {x}_{5}^{2}-{x}_{3} {x}_{5}^{2}+{x}_{4} {x}_{5}^{2}$;   \\  
\item 
$F_1 = {x}_{0}^{2} {x}_{1}-{x}_{0} {x}_{1} {x}_{2}+{x}_{1}^{2} {x}_{3}-{x}_{2}^{2} {x}_{3}+{x}_{1} {x}_{3}^{2}-{x}_{0} {x}_{2} {x}_{4}-{x}_{0} {x}_{3} {x}_{4}+{x}_{2} {x}_{4}^{2}-{x}_{0}^{2} {x}_{5}+2 {x}_{0} {x}_{1} {x}_{5}-{x}_{1}^{2} {x}_{5}+{x}_{0} {x}_{2} {x}_{5}+{x}_{2} {x}_{3} {x}_{5}-{x}_{0} {x}_{5}^{2}-{x}_{1} {x}_{5}^{2}-{x}_{3} {x}_{5}^{2}+{x}_{4} {x}_{5}^{2}$; \\ 
\item 
$F_2 = {x}_{0} {x}_{1} {x}_{2}-{x}_{1} {x}_{2}^{2}-{x}_{1}^{2} {x}_{3}-{x}_{2}^{2} {x}_{3}+{x}_{2} {x}_{3}^{2}-{x}_{0}^{2} {x}_{4}+2 {x}_{2}^{2} {x}_{4}+{x}_{1} {x}_{3} {x}_{4}+{x}_{2} {x}_{3} {x}_{4}+{x}_{0} {x}_{4}^{2}-3 {x}_{2} {x}_{4}^{2}+{x}_{1}^{2} {x}_{5}-2 {x}_{0} {x}_{4} {x}_{5}-2 {x}_{3} {x}_{4} {x}_{5}+2 {x}_{4}^{2} {x}_{5}+{x}_{0} {x}_{5}^{2}+{x}_{3} {x}_{5}^{2}-{x}_{4} {x}_{5}^{2}$;  \\  
\item 
$F_3 = 2 {x}_{0}^{2} {x}_{1}-2 {x}_{0} {x}_{1} {x}_{2}+{x}_{0} {x}_{1} {x}_{3}-{x}_{1}^{2} {x}_{3}-{x}_{0} {x}_{2} {x}_{3}-{x}_{1} {x}_{3}^{2}+{x}_{0} {x}_{1} {x}_{4}-{x}_{1} {x}_{2} {x}_{4}+{x}_{2}^{2} {x}_{4}-{x}_{0} {x}_{4}^{2}+{x}_{2} {x}_{4}^{2}-{x}_{0}^{2} {x}_{5}+{x}_{1}^{2} {x}_{5}-{x}_{0} {x}_{3} {x}_{5}+2 {x}_{1} {x}_{3} {x}_{5}+{x}_{3}^{2} {x}_{5}-{x}_{2} {x}_{4} {x}_{5}-{x}_{4}^{2} {x}_{5}+{x}_{0} {x}_{5}^{2}$;  \\  
\item\label{disjointplanes}
$F_4 = -{x}_{0}^{2} {x}_{1}+{x}_{0} {x}_{1} {x}_{2}-{x}_{1}^{2} {x}_{3}+{x}_{0} {x}_{2} {x}_{3}-3 {x}_{0}^{2} {x}_{4}+{x}_{0} {x}_{1} {x}_{4}+{x}_{2}^{2} {x}_{4}-2 {x}_{0} {x}_{3} {x}_{4}-{x}_{2} {x}_{3} {x}_{4}-{x}_{2} {x}_{4}^{2}+2 {x}_{0} {x}_{1} {x}_{5}-{x}_{1} {x}_{2} {x}_{5}+3 {x}_{0} {x}_{3} {x}_{5}+{x}_{2} {x}_{3} {x}_{5}+2 {x}_{3}^{2} {x}_{5}+{x}_{1} {x}_{4} {x}_{5}-2 {x}_{3} {x}_{4} {x}_{5}-{x}_{0} {x}_{5}^{2}$.  
\end{enumerate} 
\end{ex}
For every $i,j=0,\ldots,4$, we have the decomposition 
$
 \Sigma_i\cap V(F_j) = S \cup T_{i,j}
$,
where $T_{i,j}$ is a 
smooth rational normal scroll surface of degree $4$ 
if $(i,j)\neq(0,4)$, while 
 $T_{0,4}$ is the small variety $\langle C_1 \rangle\cup \langle C_2 \rangle\cup V({x}_{0},{x}_{3},{x}_{2} {x}_{4}-{x}_{1} {x}_{5})$. 
In particular, the smoothness of $T_{1,4}$ implies that 
there exist smooth cubic fourfolds containing a smooth quartic rational normal scroll and two disjoint planes.
The ideal of $T_{1,4}$ is generated by the following six quadratic forms: 
\begin{gather*}
{x}_{2} {x}_{4}-{x}_{1} {x}_{5},\quad  
{x}_{1} {x}_{4}+3 {x}_{4} {x}_{5}-{x}_{5}^{2},\quad  
{x}_{0} {x}_{4}-{x}_{1} {x}_{5}-{x}_{3} {x}_{5}+{x}_{4} {x}_{5},\\ 
{x}_{1} {x}_{3}-{x}_{0} {x}_{5}+{x}_{2} {x}_{5}+3 {x}_{3} {x}_{5}-{x}_{5}^{2},\ 
{x}_{1}^{2}+3 {x}_{1} {x}_{5}-{x}_{2} {x}_{5},\ 
{x}_{0} {x}_{1}-{x}_{1} {x}_{2}-{x}_{2} {x}_{3}+{x}_{1} {x}_{5} .
\end{gather*}

Moreover, if $P\subset X$ is a plane such that $P\cdot S=\tau$ with $0\leq\tau\leq 3$, then from $3h^2=S+T_{i,j}$ we deduce $P\cdot T_{i,j}=3-\tau$.

\subsection{}

Here we give some pieces of {\sc Macaulay2} code 
which have been  used to produce and verify 
 the examples above. 
 The complete code can be found in the ancillary file \texttt{cubics.m2}.
We begin by starting {\sc Macaulay2} and loading two further \href{https://github.com/Macaulay2/M2/tree/master/M2/Macaulay2/packages}{packages} included with it.
{\footnotesize
\begin{Verbatim}[commandchars=&\[\]]
Macaulay2, version 1.10
with packages: &colore[airforceblue][ConwayPolynomials], &colore[airforceblue][Elimination], &colore[airforceblue][IntegralClosure], &colore[airforceblue][InverseSystems],
               &colore[airforceblue][LLLBases], &colore[airforceblue][PrimaryDecomposition], &colore[airforceblue][ReesAlgebra], &colore[airforceblue][TangentCone]
&colore[darkorange][i1 :] &colore[airforceblue][loadPackage] "&colore[bleudefrance][Cremona]";    &colore[Sepia][-- version 4.2]
&colore[darkorange][i2 :] &colore[airforceblue][loadPackage] "&colore[bleudefrance][Resultants]"; &colore[Sepia][-- version 1.1]
\end{Verbatim}
} \noindent 
We define a method which  takes as input  a projective scheme 
 and 
 returns 
 a random smooth cubic hypersurface containing the scheme.
If such a smooth hypersurface does not exist, 
the method goes in an infinite loop and does not produce any output.
One of its possible implementations 
(that does not take care of the growth of the coefficients) is the following:
{\footnotesize
\begin{Verbatim}[commandchars=&\{\}] 
&colore{darkorange}{i3 :} randomCubic = (I) -> ( &colore{Sepia}{-- I must be a homogeneous ideal in a polynomial ring}
        B := &colore{airforceblue}{super basis}(3,&colore{airforceblue}{saturate} I); C := 0; 
        &colore{darkorchid}{while} (C == 0 &colore{darkorchid}{or} &colore{bleudefrance}{discriminant} C == 0) &colore{darkorchid}{do} C = (B * &colore{airforceblue}{random}(&colore{darkspringgreen}{QQ}^(&colore{airforceblue}{numcols} B),&colore{darkspringgreen}{QQ}^1))_(0,0);
        &colore{darkorchid}{return} C);
\end{Verbatim}
} \noindent 
The smoothness of the cubic hypersurface 
is checked through the computation of its  discriminant. 
It  is however standard to implement a general method 
 which checks whether  a given closed subscheme of a projective space 
is smooth and absolutely connected.
Now we build the parametrization $f$ 
of the quintic del Pezzo surface $S$.
{\footnotesize
\begin{Verbatim}[commandchars=&\[\]]  
&colore[darkorange][i4 :] &colore[airforceblue][use] &colore[bleudefrance][Grass](0,2,&colore[bleudefrance][Variable]=>t); 
&colore[darkorange][i5 :] P = {&colore[airforceblue][ideal](t_1,t_2),&colore[airforceblue][ideal](t_0,t_2),&colore[airforceblue][ideal](t_0,t_1),&colore[airforceblue][ideal](t_1-t_0,t_2-t_0)};
&colore[darkorange][i6 :] f = &colore[bleudefrance][rationalMap](&colore[airforceblue][intersect] P,3);
o6 = RationalMap (cubic rational map from PP^2 to PP^5)
&colore[darkorange][i7 :] S = &colore[bleudefrance][image] f; 
\end{Verbatim}
} \noindent 
The pencils $\sigma_0,\ldots,\sigma_4$, the conics $C_1,C_2$ and the points $q_1,q_2,q_3$ 
are obtained as follows:
{\footnotesize 
\begin{Verbatim}[commandchars=&\[\]] 
&colore[darkorange][i8 :] pencils = &colore[airforceblue][prepend](&colore[airforceblue][intersect] P,P); 
&colore[darkorange][i9 :] conics = (f &colore[airforceblue][ideal](pencils_0_0 + pencils_0_1), f &colore[airforceblue][ideal](pencils_0_0 + 2*pencils_0_1));
&colore[darkorange][i10 :] points = (f &colore[airforceblue][ideal](t_2,t_0-t_1), f &colore[airforceblue][ideal](t_1,t_0-t_2), f &colore[airforceblue][ideal](t_0,t_1-t_2)); 
\end{Verbatim}
} \noindent 
One then easily verifies that our choice of the planes $\Pi_0,\ldots,\Pi_3$ 
is correct. 
To determine  the Segre threefolds $\Sigma_0,\ldots,\Sigma_4$, an idea  is 
just to add another variable and to compute 
the \emph{generic} conic in each of the five pencils. We omit this code here, but it is available in 
the ancillary file. Now we produce a cubic form $F$
      like $F_4$ in Example~\ref{pfafftregub}.
The other cases are quite similar.
{\footnotesize 
\begin{Verbatim}[commandchars=&\[\]] 
&colore[darkorange][i11 :] I = &colore[airforceblue][intersect](S,&colore[airforceblue][ideal super basis](1,conics_0),&colore[airforceblue][ideal super basis](1,conics_1));
&colore[darkorange][i12 :] F = randomCubic I;
\end{Verbatim}
} \noindent 
Thus the residual intersections $T_i = \overline{\Sigma_i\cap V(F) \setminus S}$ 
can be obtained as follows (here we are assuming that 
{\footnotesize \verb!Sigma!} is the list of the ideals of $\Sigma_0,\ldots,\Sigma_4$).
{\footnotesize 
\begin{Verbatim}[commandchars=&\[\]]  
&colore[darkorange][i13 :] T = &colore[airforceblue][apply](5,i -> (Sigma_i + F):S);
\end{Verbatim}
} \noindent 
Finally, 
the following code gives us 
an explicit birational map 
from $\mathbb{P}^4$ to $V(F)$. 
{\footnotesize 
\begin{Verbatim}[commandchars=&\[\]]  
&colore[darkorange][i14 :] ((&colore[bleudefrance][rationalMap] S)|F)^-1;
o14 = RationalMap (birational map from PP^4 to hypersurface in PP^5)
\end{Verbatim}
} \noindent

\providecommand{\bysame}{\leavevmode\hbox to3em{\hrulefill}\thinspace}
\providecommand{\MR}{\relax\ifhmode\unskip\space\fi MR }
\providecommand{\MRhref}[2]{%
  \href{http://www.ams.org/mathscinet-getitem?mr=#1}{#2}
}
\providecommand{\href}[2]{#2}

\end{document}